\documentclass[11pt]{amsart}
\usepackage{times,amsmath}
\usepackage{amsthm}
\usepackage{amsmath,algorithm2e,bm}
\usepackage{amsmath,amssymb,amsfonts,amscd}
\usepackage{graphicx,latexsym,color}
\usepackage{hyperref}
\usepackage{multirow}
\usepackage{floatrow}
\usepackage{wrapfig}
\newtheorem{theorem}{Theorem}

\usepackage{verbatim}
\usepackage{subfig}

\newcommand{\p}{\partial}
\newcommand{\dif}{\,\mathrm{d}}

\newtheorem{lem}{Lemma}[section]
\def \bes{\begin{eqnarray}}
\def \ees{\end{eqnarray}}
\def \bns{\begin{eqnarray*}}
\def \ens{\end{eqnarray*}}

\newenvironment{eqa}{\begin{equation}%
  \begin{array}{rcl}}{\end{array}\end{equation}}
\newcommand\beqa{\begin{eqa}}
\newcommand\eeqa{\end{eqa}}
\newcommand{\re}[1]{\mbox{$($\ref{#1}$)$}}\baselineskip 1pt

\usepackage[bottom=1.2in,top=1in,left=1in,right=1in]{geometry}

\begin{document}
\bibliographystyle{plain}

\title{Bifurcation Analysis Reveals Solution Structures of Phase Field Models}
\author{Xinyue Evelyn Zhao}  
\address{Department of Mathematics, \\
Vanderbilt University, Nashville, TN 37212, USA} 
\email{xinyue.zhao@vanderbilt.edu}

\author{Long-Qing Chen}
\address{ 
Department of Materials Science and Engineering, Pennsylvania State University, University Park, PA, 16802, USA, corresponding author}
\email{lqc3@psu.edu}

\author{Wenrui Hao$^*$}
\address{Department of Mathematics, Pennsylvania State University, University Park, PA, 16802, USA, corresponding author}
\email{wxh64@psu.edu}

\author{Yanxiang Zhao}
\address{Department of Mathematics, The George Washington University, Washington, D.C. 20052, USA}
\email{yxzhao@gwu.edu}

\begin{abstract}
    Phase field method is playing an increasingly important role in understanding and predicting morphological evolution in materials and biological systems. Here, we develop a new analytical approach based on bifurcation analysis to explore the mathematical solution structure of phase field models. Revealing such solution structures not only is of great mathematical interest but also may provide guidance to experimentally or computationally uncover new morphological evolution phenomena in materials undergoing electronic and structural phase transitions. To elucidate the idea, we apply this analytical approach to three representative phase field equations: Allen-Cahn equation, Cahn-Hilliard equation, and Allen-Cahn-Ohta-Kawasaki system. The solution structures of these three phase field equations are also verified numerically  by the homotopy continuation method.
\end{abstract}

\maketitle

\section{Introduction}

Phase field approach is an important modeling tool for modeling interfacial evolution problems in materials science and biological systems. It is rooted in the diffuse-interface description of interfaces for fluid interfaces proposed by van der Waals more than a century ago \cite{vanderWaals_1893,Rowlinson_1979}. The early applications of the diffuse-interface to superconducting phase transitions and the compositional clustering and ordering in alloys led to the establishment of well-known time-dependent Ginzburg-Landau (TDGL) equations \cite{Ginzburg_Landau_JETP1950,Stephen_Suhl_1964}, the Cahn-Hilliard \cite{CahnHilliard_JCP1958,Cahn_1961}, and Allen-Cahn equations  \cite{AllenCahn_JP1977}  which form the basis for the evolution equations in the phase-field method. The term of ``phase-field" was coined in early applications of diffuse-interface description to solidification and dendrite growth \cite{BrowerKesslerKoplikLevine_PRA1984,Fix_1983,KesslerKoplikLevine_PRA1984,KesslerKoplikLevine_PRA1985,Boettinger_Warren_Beckerman_Karma_2002}. The generalization of the phase-field to include both physical and artificial fields to distinguish different phases has led to wide-spread applications of the phase-field method to modeling morphological and microstructure evolution in a wide variety processes beyond solidification in materials science \cite{Chen_Zhao_2022,Chen_ARMR2002,Steinbach_2013}, biology \cite{Aranson2016},fluid and solid mechanics \cite{AndersonMcFaddenWheeler_ARFM1998,Borden_CMAME2012}, etc.

In the phase-field method, one introduces a labeling function, called phase-field, $\phi$ which for a two-phase system, is assigned a value (say, -1) for one phase, and another value (say, +1) for the other. In the interfacial region, the phase field labeling function $\phi$ rapidly but smoothly transitions from -1 to 1. Meanwhile, the interface is tracked by a level set, typically the 0-level set, during the morphological evolution. The main advantage of the phase field approach is that it can predict the evolution of arbitrary morphologies and complex microstructures without explicitly tracking interfaces, and thereby easily handle topological changes of interfaces.

Since phase-field method involves the numerical solutions to systems of partial differential equations in time and space, there have been extensive efforts in developing numerical methods for solving the phase-field equations. For example, several classic methods such as explicit/implicit Euler methods, Crank-Nicolson and its variant, linear multistep methods, and Runge-Kutta methods have been considered for the time discretization (See for example \cite{ChenShen_CPC1998,LiuShen_PhysicaD,DuNicolaides_SINU1991,Akrivis_MCAMS1998,SongShu_ISC2017} and the references cited therein). For the spatial discretization, methods such as finite difference, finite element, discontinuous Galerkin, and spectral approximation are typical examples {See the recent review article \cite{DuFeng_HNA2020} and the reference cited therein for more detailed discussion). Some stabilized schemes have recently been developed based on the inheriting the energy dissipation law and phase field gradient flow dynamics, including the convex splitting \cite{Eyre_MRS1998}, linearly-implicit stabilized schemes \cite{XuZhao_JSC2019, Xu_CMAME2019}, exponential integrator \cite{CoxMatthews_JCP2002,DuJuLiQiao_SIAMReview2021}, Invariant energy quadratization \cite{Yang_JCP2016}, and scalar auxiliary variable schemes \cite{ShenXuYang_JCP2018}. 

However, there are very limited studies to reveal the solution structure of phase field models which is critical to understanding the models from a mathematical point of view and to exploring the possible formation of novel morphological patterns. Solution structures of nonlinear differential equations have been well-studied by exploring bifurcations \cite{Rhein1} and multiple solutions \cite{BKHR}. Existing theories and numerical
methods have contributed to a better understanding of these solution structures and the relationship between solutions and parameters \cite{haber1}. For example, the Crandall-Rabinowitz theorem has been used to theoretically study the bifurcations  of nonlinear differential equations such as free boundary problems \cite{FRF1,FHB,FHB1,zhao2020symmetry,zhao2021bifurcation}.
Numerically, homotopy continuation method \cite{hao2020adaptive,MorganSommese1} has been successfully employed to study parametric problems such as bifurcation
\cite{Rhein1} and the structural stability
\cite{Rhein2}. Recently, several numerical
methods have been developed based on homotopy continuation methods for computing multiple solutions, steady-states, and bifurcation points of nonlinear PDEs \cite{HHHLSZ,HHHS}. These numerical methods have been also applied to hyperbolic conservation laws \cite{HHSSXZ}, physical systems \cite{HNS,HNS1} and some more complex free boundary problems arising from biology \cite{HCF,HF}. 

In this paper, we develop an analytical framework to study the solution structure of phase field models and apply it to three well-known phase field equations. In particular, we study the solution structure of Allen-Cahn equation in Section 3 and that of Cahn-Hillard equation in Section 4. Finally in Section 5, we discuss the solution structure of Allen-Cahn-Ohta-Kawazaki system which is used to model the morphology of diblock copolymer systems.

\section{Bifurcation analysis and homotopy tracking}

The goal of this paper is to compute global bifurcation diagram for various PDE models. Our approach combine analytical and numerical method. First, we analyze the bifurcations of various phase field models from the trivial steady-states by Crandall-Rabinowitz theorem, then we numerically compute the global bifurcation diagram via homotopy tracking. Generally speaking, we consider the following nonlinear operator 
\[\mathcal{F}(x,\mu)=0,\]
where $\mathcal{F}(\cdot,\mu)$ is a  $C^p$ map, $p\ge 1$ from a real Banach space  $X$ to another real Banach space $Y$ and $\mu\in \mathbb{R}$ is a parameter. The bifurcation of $x$ with respect to the parameter $\mu$ can be verified theoretically by  Crandall-Rabinowitz theorem \cite{crandall1971bifurcation}. 

\begin{theorem}\label{CRthm}  {\bf (Crandall-Rabinowitz theorem, \cite{crandall1971bifurcation})}
Let $X$, $Y$ be real Banach spaces and $\mathcal{F}(\cdot,\cdot)$ be a $C^p$ map of a neighborhood $(0,\mu_0)$ in $X \times \mathbb{R}$ into $Y$. Denote by $D_x \mathcal{F}$ and $D_{\mu x}\mathcal{F}$ the first- and second-order Fr\'echet derivatives, respectively. 
Assume the following four conditions hold:
\begin{itemize}
\item[(I)] $\mathcal{F}(0,\mu) = 0$ for all $\mu$ in a neighborhood of $\mu_0$,
\item[(II)] $\mathrm{Ker} \,D_x\mathcal{F}(0,\mu_0)$ is one dimensional space, spanned by $x_0$, 
\item[(III)] $\mathrm{Im} \,D_x \mathcal{F}(0,\mu_0)=Y_1$ has codimension 1,
\item[(IV)] $D_{\mu x} \mathcal{F}(0,\mu_0) x_0 \notin Y_1$,
\end{itemize}
then $(0,\mu_0)$ is a bifurcation point of the equation $\mathcal{F}(x,\mu)=0$ in the following sense: in a neighborhood of $(0,\mu_0)$, the set of solutions $\mathcal{F}(x,\mu) =0$ consists of two $C^{p-2}$ smooth curves, $\Gamma_1$ and $\Gamma_2$, which intersect only at the point $(0,\mu_0)$; $\Gamma_1$ is the curve $(0,\mu)$, and $\Gamma_2$ can be parameterized as follows:
$$\Gamma_2: (x(\epsilon),\mu(\epsilon)), |\epsilon| \text{ small, } (x(0),\mu(0))=(0,\mu_0),\; x'(0)=x_0.$$
\end{theorem}

Although the bifurcation theory can help in some
special cases, the in-depth study of solution structures often requires numerical methods to derive bifurcation diagrams of nonlinear systems. Generally speaking, the  nonlinear operator $\mathcal{F}$ is approximated by $\mathcal{F}^h$ numerically ($h$ refers to the mesh size of numerical discretization). Then the numerical solution $x^h$ is computed by solving the following discretized nonlinear system:  

\begin{equation}\label{Sys}
	\mathcal{F}^h({x}^h,\mu)=\mathbf{0},
\end{equation}
where $\mathcal{F}^h:
\mathbb{R}^n\times\mathbb{R}\rightarrow\mathbb{R}^n$
and  $x^h$ is the variable vector that depends on the parameter $\mu$, i.e., $x^h=x^h(\mu)$. Suppose we have a solution at the starting point, namely $x^h(\mu_0)=x_0$, various homotopy tracking algorithms can be used to compute the solution path \cite{hao2021adaptive,hao2020adaptive,hao2021stochastic}, $x^h(\mu)$. If $\p_x\mathcal{F}^{h}(x^h,\mu)$ is nonsingular, the solution path $x^h(\mu)$ is smooth and unique. However, when $\p_x\mathcal{F}^{h}(x^h,\mu)$ becomes singular, the solution path hits the singularity and different types of bifurcations are formed \cite{hao2020spatial}. 

More specifically, the homotopy tracking algorithm consists of a predictor step and a corrector step to solve the parametric problem.
 The predictor is to compute the solution at
	$\mu_1=\mu_0+\Delta \mu$ by setting \[\mathcal{F}^h(x_0+\Delta x^h,\mu_0+\Delta
	\mu)=0,\] which, at the first order, yields an Euler predictor,\[
	\p_{x}\mathcal{F}^{h}(x_0,\mu_0)\Delta x^h=-\p_{\mu}\mathcal{F}^{h}(x_0,\mu_0)\Delta 
	\mu.
	\]
Then we apply
 Newton corrector to refine the solution with an initial guess $\widetilde{x}^h=x_0+\Delta x^h$:
\[	\p_{x}\mathcal{F}^{h}(\widetilde{x}^h,\mu_1)\Delta
	x^h=-\mathcal{F}^{h}(\widetilde{x}^h,\mu_1),
\] 
and repeat $\widetilde{x}^h=\widetilde{x}^h+\Delta x^h$ until $(\widetilde{x}^h,\mu_1)$ is on the path, namely, $\mathcal{F}^h(\widetilde{x}^h,\mu_1)=0$.

\section{Bifurcation analysis of Allen-Cahn equation}

In this and the following sections, we will consider two classical phase field equations: Allen-Cahn equation and Cahn-Hilliard equation. As a mathematical convention, we take $\phi = \pm1$ in two distinct phases, respectively. This is in contrast to the Allen-Cahn-Ohta-Kawasaki equation in Section 5, in which we rather take $\phi = 0$ or $1$ in the two phases from physical perspective.

We consider Allen-Cahn equation
\bes \label{ACeqn}
\frac{\partial \phi}{\partial t}(\mathbf{x}, t)=\epsilon \Delta \phi(\mathbf{x}, t)- \frac{1}{\epsilon}W^{\prime}(\phi(\mathbf{x}, t)), \quad \mathbf{x} \in \Omega, t>0.
\ees
Here $\Omega = [-1, 1]^d, d =1, 2, 3$, and $0<\epsilon \ll 1$ is a parameter to control the width of the interface. $\phi$ is a phase field labeling function which equals $\pm1$ in two distinct phases. The function $W(\phi) = \frac{1}{4}(\phi^2-1)^2$ is a double well potential which enforces the phase field function $\phi$ to be equal to 1 inside the interface and -1 outside the interface. The Allen-Cahn equation (\ref{ACeqn}) can be viewed as the $L^2$ gradient flow dynamics for the Ginzburg-Landau free energy functional
\begin{align}\label{eqn:GL}
    E(\phi) = \int_{\Omega} \frac{\epsilon}{2}|\nabla\phi|^2 + \frac{1}{\epsilon}W(\phi)\ \text{d}\mathbf{x}.
\end{align}
In 1D case, Ginzburg-Landau free energy reduces to
\bes
\int_{-1}^1 \frac{\epsilon}{2}(\phi_x)^2+\frac{1}{4\epsilon}\big(\phi^2-1\big)^2\dif x,
\ees
and the associated Euler-Lagrange equation (steady-state Allen-Cahn equation) becomes
\begin{equation}\label{eqn:steadyAllenCahn}
\begin{cases}
&-\epsilon\phi_{xx} + \frac{1}{\epsilon} (\phi^3 - \phi) = 0 \hspace{2em} -1<x<1,\\
&\phi_x(-1) = \phi_x(1) = 0.
\end{cases}
\end{equation}

\subsection{Bifurcation analysis}
It is easy to verify that $\phi\equiv\phi_0= -1,0,1$ are three trivial solutions of the steady-state system (\ref{eqn:steadyAllenCahn}). Besides these trivial solutions, we are more interested in non-trivial steady states, which can bifurcate from the zero trivial steady-state. More specifically, we consider the following shifted system from $\phi_0$:
\begin{equation}\label{eqn:shiftedSteadyAllenCahn}
    \begin{cases}
        &-\phi_{xx} + \frac{1}{\epsilon^2}[(\phi+\phi_0)^3 - (\phi+\phi_0)] = 0 \hspace{2em} -1<x<1,\\
        &\phi_x(-1)=\phi_x(1)=0.
    \end{cases}
\end{equation}
We can verify that $\phi=0$ is always a solution to the system \re{eqn:shiftedSteadyAllenCahn}.

Next, we consider the following Banach space:
\begin{equation}\label{Banachsp}
    X^{l+\alpha} = \{\phi(x)\in C^{l+\alpha}[-1,1], \phi_x(-1)=\phi_x(1)=0\},
\end{equation}
with the H\"{o}lder norm
\begin{equation*}
    \|u\|_{X^{l+\alpha}} = \|u\|_{C^l([-1,1])} + \max\limits_{|\beta|=l} |D^\beta u|_{C^\alpha([-1,1])},
\end{equation*}
where $l\ge0$ is an integer, $0<\alpha <1$, and 
\begin{equation*}
    |u|_{C^\alpha([-1,1])} = \sup \limits_{x\neq y \,\in(-1,1)} \frac{|u(x)-u(y)|}{|x-y|^\alpha}. 
\end{equation*}
Taking
\begin{equation}
    \label{spaces}
    X=X^{l+2+\alpha} \hspace{2em}\text{and}\hspace{2em} Y=X^{l+\alpha}
\end{equation}
in Crandall-Rabinowitz Theorem (Theorem \ref{CRthm}), and defining an operator $\mathcal{F}$ as
\begin{equation}
    \label{operatorF}
    \mathcal{F}(\phi,\epsilon) = -\phi_{xx} + \frac{1}{\epsilon^2} [(\phi+\phi_0)^3 - (\phi+\phi_0)],
\end{equation}
where $\epsilon$ is the bifurcation parameter, we know that $\mathcal{F}(\cdot, \epsilon)$ maps $X$ into $Y$. 

Since $\phi=0$ is always a solution to the system (\ref{eqn:shiftedSteadyAllenCahn}), $\mathcal{F}(0,\epsilon)=0$ for every $\epsilon$, and Condition I of Crandall-Rabinowitz Theorem (Theorem \ref{CRthm}) is satisfied. To verify other conditions, we need to compute the Fr\'echet derivative of the operator $\mathcal{F}$, which is given in the following lemma.

\begin{lem}\label{lem1.1}
The Fr\'echet derivative $D_\phi \mathcal{F}(\phi,\epsilon)$ of the operator $\mathcal{F}$ is given by
\begin{equation}\label{freD}
    D_\phi \mathcal{F}(\phi,\epsilon)[\xi] =-\xi_{xx} + \frac{1}{\epsilon^2}[3 (\phi+\phi_0)^2 \xi - \xi].
\end{equation}
\end{lem}
\begin{proof}
By taking $\phi,\xi\in X$,  we have
\begin{equation*}
    \begin{split}
        &\mathcal{F}(\phi+\xi,\epsilon)-\mathcal{F}(\phi,\epsilon) - D_\phi \mathcal{F}(\phi,\epsilon)[\xi]\\ 
        &= \frac{1}{\epsilon^2}[(\phi+\phi_0+\xi)^3 - (\phi+\phi_0+\xi) - (\phi+\phi_0)^3 + (\phi+\phi_0) - 3(\phi+\phi_0)^2 \xi + \xi]\\
        &= \frac{1}{\epsilon^2} \xi^2[\xi+3(\phi+\phi_0)].
    \end{split}
\end{equation*}
Therefore,
\begin{equation*}
\begin{split}
    \frac{\|\mathcal{F}(\phi+\xi,\epsilon)-\mathcal{F}(\phi,\epsilon) - D_\phi \mathcal{F}(\phi,\epsilon)\xi\|_Y}{\|\xi\|_X} &= \frac{\|\frac{1}{\epsilon^2} \xi^2[\xi+3(\phi+\phi_0)]\|_Y}{\|\xi\|_X}\\
    &\le \frac{1}{\epsilon^2} \frac{\|\xi\|_X^2\|\xi+3(\phi+\phi_0)\|_Y}{\|\xi\|_X} \rightarrow 0,
    \end{split}
\end{equation*}
as $\|\xi\|_X \rightarrow 0$. Hence, $D_\phi \mathcal{F}(\phi,\epsilon)$ is the Fr\'echet derivative of $\mathcal{F}$.
\end{proof}

Given \re{freD}, we have by taking $\phi=0$ that
\begin{equation}
    \label{FreD0}
    D_\phi \mathcal{F}(0,\epsilon)[\xi] =-\xi_{xx} +\frac{1}{\epsilon^2}(3\phi_0^2\xi - \xi).
\end{equation}
 For condition II of Crandall-Rabinowitz Theorem, we need to analyze the structure of $\text{Ker} \big(D_\phi\mathcal{F}(0,\epsilon)\big)$ which is given by $D_\phi \mathcal{F}(0,\epsilon)[\xi] = 0$, $\xi\in X$. By \re{FreD0}, it is equivalent to solve the following system
\begin{equation}\label{ker}
    \begin{cases}
        &-\xi_{xx} + \frac{1}{\epsilon^2}(3\phi_0^2 \xi -\xi) = 0,\\
        &\xi_x(-1)=\xi_x(1)=0.
    \end{cases}
\end{equation}
This system can be solved by using an eigenfunction ansatz, i.e., 
\begin{equation}\label{eigen}
    \xi(x) = a_0 + \sum_{n=1}^\infty a_n\cos(Lnx) + \sum_{n=1}^\infty b_n\sin(Lnx),
\end{equation}
where $L$ is to be determined. By taking the derivative, we have
\begin{equation}
    \xi_x(x) = -\sum_{n=1}^\infty a_n Ln \sin(Lnx) + \sum_{n=1}^\infty b_n Ln \cos(Lnx).
\end{equation}
The two boundary conditions, $\xi_x(1)=0$ and $\xi_x(-1)=0$, yield that either $a_n=0$ or $b_n=0$.  If $a_n=0$, then $Ln=\frac{\pi}{2}+(n-1)\pi$, hence 
\begin{equation}
    \label{sol1}
    \xi(x) = a_0 + \sum_{n=0}^\infty b_n \sin\left(\left(\frac{\pi}{2}+n\pi\right)x \right);
\end{equation}
if $b_n=0$, then $Ln=n\pi$, and we have
\begin{equation}
    \label{sol2}
    \xi(x)=a_0 + \sum_{n=1}^\infty a_n \cos(n\pi x) = \sum_{n=0}^\infty a_n\cos(n\pi x).
\end{equation}

Next, we will determine the bifurcations with respect to  $\epsilon$ by solving \re{ker} with different trivial solution $\phi_0$.

\subsubsection{Bifurcations around $\phi_0=0$}
\begin{theorem}\label{bif1}
For each integer $n\ge 0$, $\epsilon_n^{(1)} = \frac{1}{\pi/2+n\pi}$, $(0,\epsilon_n^{(1)})$ is a bifurcation point to the system \re{eqn:shiftedSteadyAllenCahn} such that there is a bifurcation solution $(\phi_n(x,s),\epsilon_n(s))$ with
$$\epsilon_n(s)=\epsilon_n^{(1)}+s,\hspace{2em} \phi_n(x,s)=s\sin\left(\left(\frac{\pi}{2}+n\pi\right)x\right)+O(s^2),\hspace{2em}\text{where $|s|\ll1$.}$$
\end{theorem}
\begin{proof}
We need to verify the four conditions of Crandall-Rabinowitz Theorem at $(0,\epsilon_n^{(1)})$. In this case, we use the Fourier expansion of $\xi(x)$ in \re{sol1}. 
When $\phi_0=0$, it follows from \re{FreD0} that
\begin{equation*}
    D_\phi \mathcal{F}(0,\epsilon) [\xi(x)] =-\frac{a_0}{\epsilon^2} + \sum_{n=0}^\infty b_n\left[\left(\frac{\pi}{2}+n\pi\right)^2 - \frac{1}{\epsilon^2} \right] \sin\left(\left(\frac{\pi}{2}+n\pi\right)x\right).
\end{equation*}
If $\epsilon = \epsilon_n^{(1)} = \frac{1}{\pi/2+n\pi}$, the term with $\sin\big((\frac{\pi}{2}+n\pi)x\big)$ disappears while all the other terms remain. Hence, we have
\begin{align*}
    &D_\phi \mathcal{F}\left(0,\epsilon_n^{(1)}\right) \left[b_n\sin\left(\left(\frac{\pi}{2}+n\pi\right)x\right)\right] = b_n\left[\left(\frac{\pi}{2}+n\pi\right)^2 -\left(\epsilon_n^{(1)}\right)^{-2} \right] \sin\left(\left(\frac{\pi}{2}+n\pi\right)x\right) = 0, \\
    &D_\phi \mathcal{F}(0,\epsilon_n^{(1)}) [\xi(x)] =-\frac{a_0}{\epsilon^2} + \sum_{\substack{k=0 \\ k\neq n}}^\infty b_k\left[\left(\frac{\pi}{2}+k\pi\right)^2 - \left(\epsilon_n^{(1)}\right)^{-2} \right] \sin\left(\left(\frac{\pi}{2}+k\pi\right)x\right).
\end{align*}
It follows from the above two equations that
\begin{equation*}
    \text{Im}\ D_\phi \mathcal{F}\left(0,\epsilon_n^{(1)}\right) = \text{span}\Big\{1, \sin\left(\frac{\pi}{2}x\right),\cdots,\sin\left(\left(\frac{\pi}2 +(n-1)\right)x\right), \sin\left(\left(\frac{\pi}2 +(n+1)\right)x\right),\cdots\Big\},
\end{equation*}
and
\begin{equation*}
    \text{Ker}\ D_\phi \mathcal{F}\left(0,\epsilon_n^{(1)}\right)  = \text{span}\Big\{\sin\left(\left(\frac{\pi}{2}+n\pi\right)x\right) \Big\},
\end{equation*}
which indicate that $\text{dim}\left(\text{Ker}\,D_\phi \mathcal{F}\left(0,\epsilon_n^{(1)}\right)\right) = 1$ and $\text{codim}\left(\text{Im}\, D_\phi \mathcal{F}\left(0,\epsilon_n^{(1)}\right)\right)=1$. Finally, by differentiating \re{FreD0} with respect to $\epsilon$, substituting into $\phi_0=0$, and applying the operator on $\sin\big((\frac{\pi}{2}+n\pi)x\big)$, we obtain
\begin{equation*}
    D_{\epsilon \phi} \mathcal{F}\left(0,\epsilon_n^{(1)}\right) \left[\sin\left(\left(\frac{\pi}{2}+n\pi\right)x\right)\right] = 2\left(\epsilon_n^{(1)}\right)^{-3} \sin\left(\left(\frac{\pi}{2}+n\pi\right)x\right)
   \; \notin \text{ Im}\,D_\phi \mathcal{F}\left(0,\epsilon_n^{(1)}\right).
\end{equation*}
Therefore, all four conditions of Crandall-Rabinowitz Theorem are satisfied, and the proof is completed.
\end{proof}

In the similar manner, using the solution expression \re{sol2}, we obtain the following result.
\begin{theorem}\label{bif2}
For each integer $n\ge1$ and $\epsilon_n^{(2)} = \frac{1}{n\pi}$, $\left(0,\epsilon_n^{(2)}\right)$ is a bifurcation point to the system \re{eqn:shiftedSteadyAllenCahn} such that there is a bifurcation solution $(\phi_n(x,s),\epsilon_n(s))$ with
$$\epsilon_n(s)=\epsilon_n^{(2)}+s,\hspace{2em} \phi_n(x,s)=s\cos(n\pi x)+O(s^2),\hspace{2em}\text{where $|s|\ll1$.}$$
\end{theorem}

\subsubsection{No bifurcations around  $\phi_0=\pm1$}
When $\phi_0 = \pm1$, we derive from \re{FreD0} that $D_\phi \mathcal{F}(0,\epsilon)[\xi] = -\xi_{xx}+\frac{2}{\epsilon^2}\xi$. Letting it equal to 0, we have
$$-\xi_{xx} + \frac{2}{\epsilon^2}\xi = 0.$$
We multiply both sides with $\xi$ and integrate them over $[-1,1]$. From integration by parts, one can obtain
\begin{equation*}
    \int_{-1}^1 \left(|\xi_x|^2 + |\xi|^2\right) \dif x = 0,
\end{equation*}
which yields $\xi\equiv0$ in this case. Therefore, there are no bifurcation solutions around the constant solutions $\phi\equiv\phi_0=\pm1$. Alternatively, one can also use the Crandall-Rabinowitz Therorem to show that there are no bifurcation solutions. To this end, we substitute the solution expression \re{sol2} into $D_\phi \mathcal{F}(0,\epsilon)[\xi]$ to obtain
\begin{equation*}
    D_\phi \mathcal{F}(0,\epsilon) [\xi(x)] =\frac{2 a_0}{\epsilon^2} + \sum_{n=1}^\infty a_n\left[(n\pi)^2 + \frac{2}{\epsilon^2} \right] \cos(n\pi x).
\end{equation*}
The coefficients $[(n\pi)^2+\frac{2}{\epsilon^2}]$ are always positive, so it is impossible to make any term of $\cos(n\pi x)$ varnish. Therefore, we cannot find any value of $\epsilon$ to meet the four conditions of the Crandall-Rabinowitz Theorem. The same conclusion holds if we use the solution expression \re{sol1}. Hence, it also proves that there are no bifurcation points around the constant solutions $\phi\equiv\phi_0=\pm1$.

\subsection{Bifurcation diagram}
Guided by Theorems \ref{bif1} and \ref{bif2}, we consider the bifurcation points along the trivial solution branch $\phi=0$.  First, we discretized Allen-Cahn equation (\ref{eqn:steadyAllenCahn}) by using the finite difference method, namely,
\begin{equation}\label{AC_FD}
\begin{cases}
&-\frac{\phi_{i+1}+\phi_{i-1}-2\phi_{i}}{h^2} + \frac{1}{\epsilon} (\phi^3_i - \phi_i) = 0 \hspace{2em} 0\leq i\leq N,
\end{cases}
\end{equation}
where $h=2/N$, $N$ is the number of grid points, and $\phi_i$ is the numerical approximation of $\phi(-1+hi)$.
Moreover, we introduce two ghost points to ensure the boundary condition, namely, $\phi_{-1} =\phi_1, \phi_{N+1} = \phi_N$. By setting $N=200$, we compute the bifurcation diagram of 1D Allen-Cahn equation via the homotopy tracking and show the results in Fig. \ref{Fig:Bifurcation} for $n\leq 10$. 
Note that the bifurcation points of the two eigenfunctions in \re{sol1} and \re{sol2} alternate in the diagram. 

For any given $\epsilon$, multiple solutions can be computed from the bifurcation diagram. For instance, we have $12$ non-trivial solutions for $\epsilon=0.1$ shown in Fig. \ref{Fig:ms}. 

\begin{figure}
  \includegraphics[width=6in]{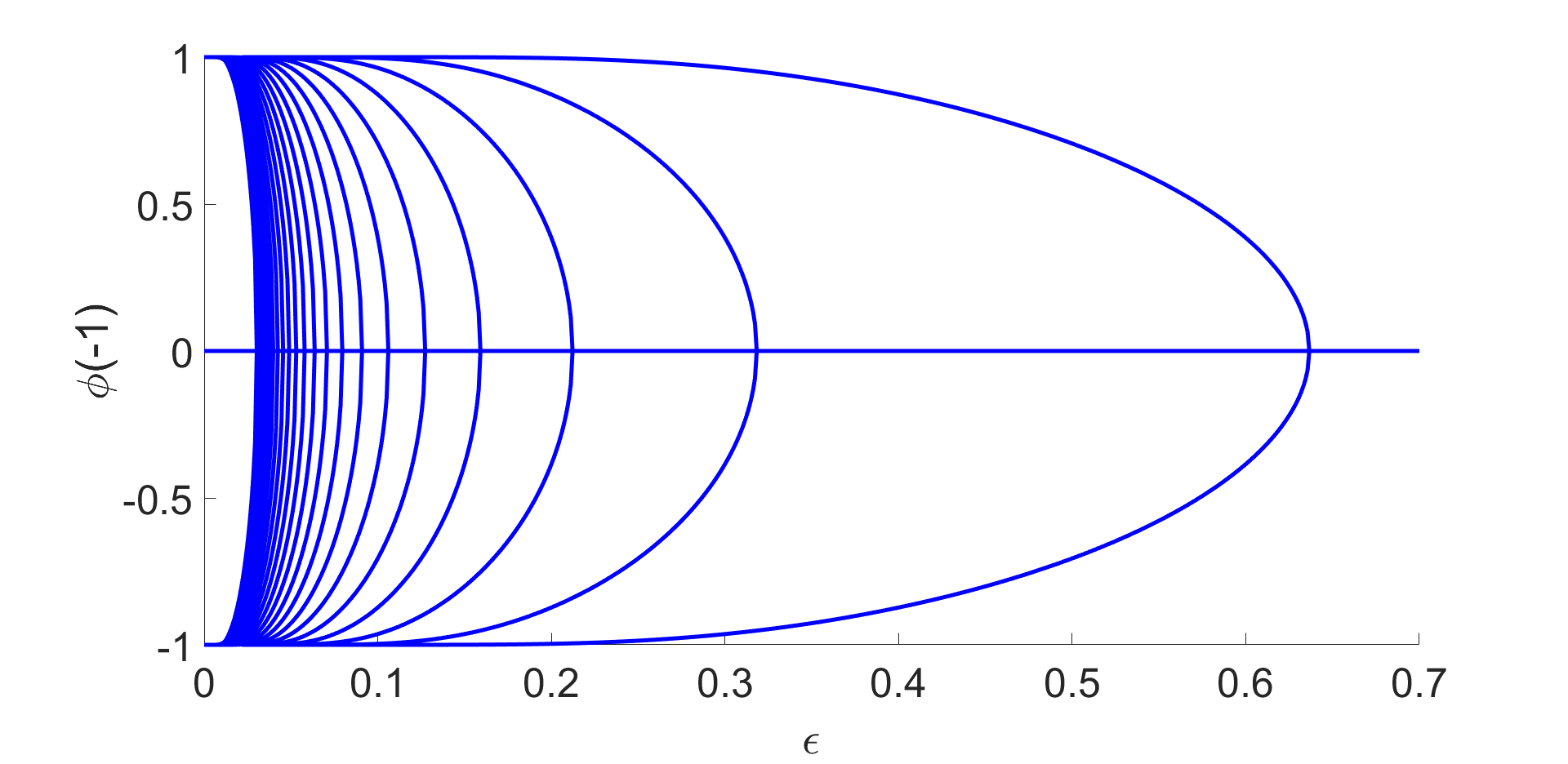}
  \caption{The bifurcation diagram of the 1D Allen-Cahn equation v.s. $\epsilon$ for $n\leq 10$.}\label{Fig:Bifurcation}
\end{figure}

\begin{figure}
  \centering
  \includegraphics[width=6in]{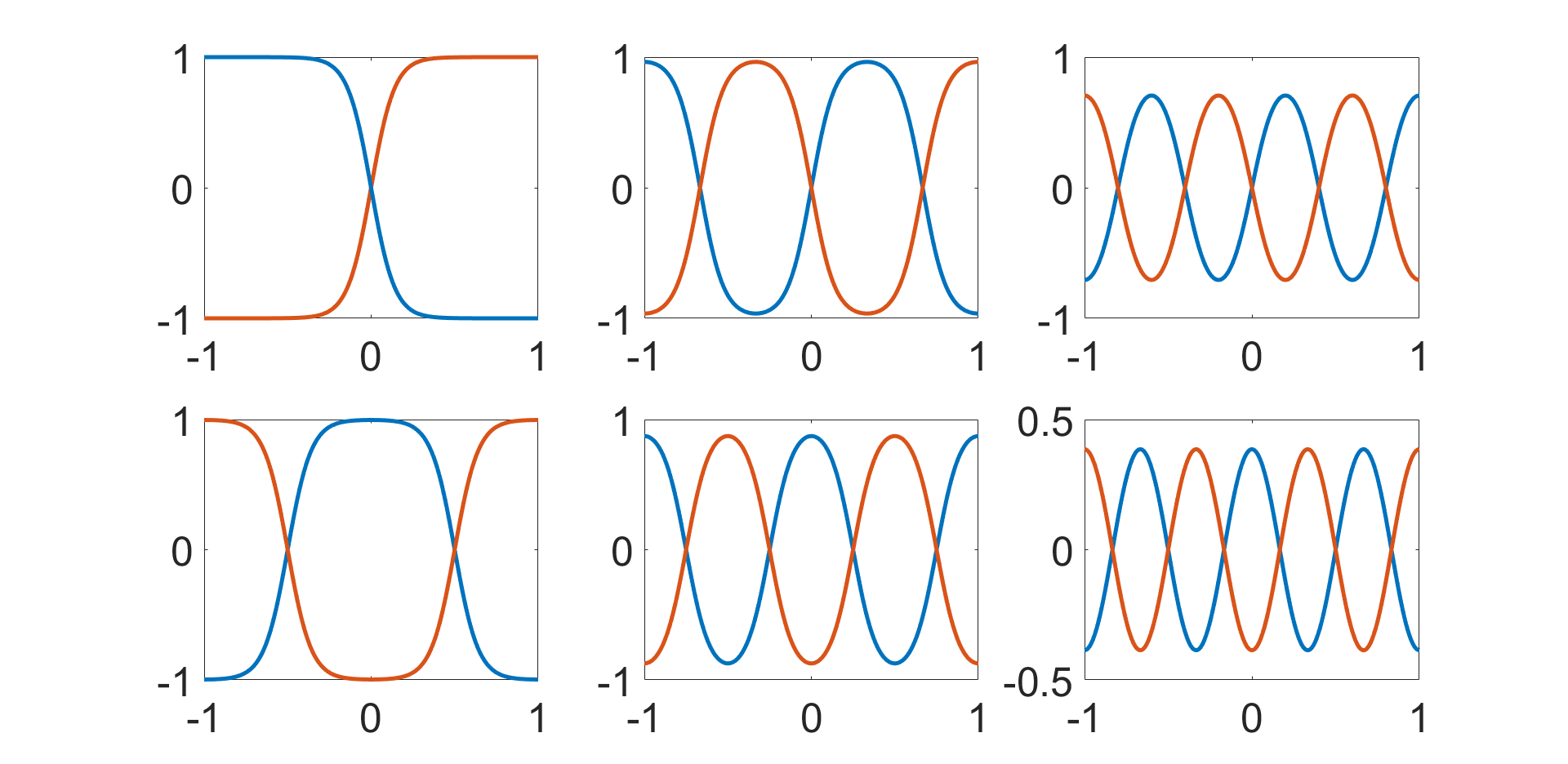}
  \caption{The 12 non-trivial solutions of the 1D Allen-Cahn equation for $\epsilon=0.1$}\label{Fig:ms}
\end{figure}

\section{Cahn–Hilliard equation}
Next, we consider Cahn-Hilliard (CH) equation:
\bes
\begin{aligned}
\frac{\partial \phi}{\partial t}(\mathbf{x}, t) &=\Delta \mu(\mathbf{x}, t), \\
\mu(\mathbf{x}, t) &= -\epsilon \Delta \phi(\mathbf{x}, t) + \frac{1}{\epsilon} W^{\prime}(\phi(\mathbf{x}, t)),
\end{aligned}
\ees
where $\mu$ is the chemical potential of the system. CH equations arise as a phenomenological model for isothermal phase separation and can be viewed as the $H^{-1}$ gradient flow dynamics of Ginzburg-Landau free energy (\ref{eqn:GL}).
The 1D steady state system of CH equation reads
\begin{equation}
    \label{1dCH}
    \begin{cases}
    & \mu_{xx} = 0,\\
    &-\phi_{xx}+\frac{1}{\epsilon^2}(\phi^3-\phi)=\mu(x),\\
    &\phi_x(-1)=\phi_x(1)=0,
    \end{cases}
\end{equation}
For simplicity, we only consider the case where $\mu(x)\equiv\mu_0$ is a constant. 

We start with finding trivial steady states $\phi\equiv \phi_0$ for the system \re{1dCH}. The trivial steady state satisfies
\begin{equation}
    \label{CH1}
    \phi_0^3 - \phi_0 = \mu_0\epsilon^2.
\end{equation}
If $-\frac{2}{3\sqrt{3}}<\mu_0\epsilon^2<\frac{2}{3\sqrt{3}}$, the equation \re{CH1} admits 3 real solutions. When $\mu_0=0$, the middle solution is $\phi_0 = 0$; in addition, in the case where there are 3 real solutions, the middle solution takes an approximation form as
\begin{equation}
    \label{CH:sol}
    \phi_0 = \phi_0(\mu_0,\epsilon) = 0 - \mu_0\epsilon^2 + O(\mu_0^2 \epsilon^4),
\end{equation}
when $|\mu_0 \epsilon^2|\ll 1$. 

The bifurcation analysis for Cahn-Hilliard equation is similar to that for Allen-Cahn equation in the last section. To begin with, we shift the solution to \re{1dCH} by letting $\widetilde{\phi} = \phi - \phi_0(\mu_0,\epsilon)$. After dropping the $\,\widetilde{}\,$ notation, we obtain the new system as 
\begin{equation}\label{CH}
    \begin{cases}
        &-\phi_{xx} + \frac{1}{\epsilon^2} \Big[(\phi+\phi_0(\mu_0,\epsilon))^3 - (\phi+\phi_0(\mu_0,\epsilon))\Big] = \mu_0,\\
        &\phi_x(-1) = \phi_x(1) = 0.
    \end{cases}
\end{equation}

Similar as in the last section, we set two Banach spaces in Crandall-Rabinowitz Theorem as $X=X^{l+2+\alpha}$ and $Y=X^{l+\alpha}$, where the space $X^{l+\alpha}$ is defined in \re{Banachsp}. In addition, we define the following operator
\begin{equation}
    \label{CH:F}
    \widetilde{\mathcal{F}}(\phi,\epsilon) = -\phi_{xx} + \frac{1}{\epsilon^2} \Big[(\phi+\phi_0(\mu_0,\epsilon))^3 - (\phi+\phi_0(\mu_0,\epsilon)\Big] - \mu_0.
\end{equation}
By Lemma \ref{lem1.1}, it is easy to compute the Fr\'echet derivative of $\widetilde{\mathcal{F}}(\phi,\epsilon)$,
\begin{equation}
    \label{CH:FD}
    D_\phi \widetilde{\mathcal{F}}(\phi,\epsilon)[\xi] = -\xi_{xx} + \frac{1}{\epsilon^2}[3(\phi+\phi_0(\mu_0,\epsilon))^2 \xi - \xi].
\end{equation}
Noticing that $D_\phi \widetilde{\mathcal{F}}(\phi,\epsilon)[\xi]$ takes the same form as $D_\phi \mathcal{F}(\phi,\epsilon)[\xi]$, the bifurcation analysis is along similar lines. Since $\phi=0$ is always a solution to \re{CH}, we substitute $\phi=0$ in \re{CH:FD} to obtain
\begin{equation}
    \label{CH:FD0}
    D_\phi \widetilde{\mathcal{F}}(0,\epsilon)[\xi] = -\xi_{xx} + \frac{1}{\epsilon^2}[3\phi_0(\mu_0,\epsilon)^2 \xi - \xi].
\end{equation}
If $\xi(x)$ takes the form of \re{sol1}, we have
\begin{equation*}
    D_\phi \widetilde{\mathcal{F}}(0,\epsilon)[\xi] = \frac{3\phi_0(\mu_0,\epsilon)^2 - 1}{\epsilon^2}a_0 + \sum_{n=0}^\infty b_n\left[\left(\frac{\pi}{2}+n\pi\right)^2 + \frac{3\phi_0(\mu_0,\epsilon)^2 -1}{\epsilon^2} \right] \sin\left(\left(\frac{\pi}{2}+n\pi\right)x\right).
\end{equation*}
The term of $ \sin\big((\frac{\pi}{2}+n\pi)x\big)$ becomes zero if and only if
\begin{equation}
    \label{CH2}
    \left(\frac{\pi}{2}+n\pi\right)^2 + \frac{3\phi_0(\mu_0,\epsilon)^2 -1}{\epsilon^2} = 0,
\end{equation}
hence we require $3\phi_0(\mu_0,\epsilon)^2 - 1 < 0$. Combining with \re{CH1}, we know that when $-\frac{2}{3\sqrt{3}}< \mu_0\epsilon^2 < \frac{2}{3\sqrt{3}}$, the middle solution satisfies the condition $3\phi_0(\mu_0,\epsilon)^2 - 1 < 0$. Therefore, all the non-trivial solutions bifurcate from the middle solution when $-\frac{2}{3\sqrt{3}}< \mu_0\epsilon^2 < \frac{2}{3\sqrt{3}}$, and in the following analysis, we use $\phi_0(\mu_0,\epsilon)$ to exclusively represent the middle solution.

To further locate the bifurcation points, we need to solve $\epsilon$ from \re{CH2}, and it is equivalent to solve $\mathcal{H}_n(\mu_0,\epsilon)=0$, where
\begin{equation}
    \label{CH:H}
    \mathcal{H}_n(\mu_0,\epsilon) = \left(\frac{\pi}{2}+n\pi\right)^2 - \frac{1}{\epsilon^2} + \frac{3\phi_0(\mu_0,\epsilon)^2}{\epsilon^2}.
\end{equation}
For an integer $n\ge 0$, we denote the solution  by $\epsilon=\epsilon_n$. First, we consider the case when $\mu_0=0$, in which the middle solution $\phi_0(0,\epsilon)=0$. Substituting it into \re{CH:H}, we derive the solution to $\mathcal{H}_n(0,\epsilon)=0$ is 
\begin{equation*}
    \epsilon=\epsilon_n^0 = \frac{1}{\pi/2 + n\pi}.
\end{equation*}
When $\mu_0\neq 0$ and $|\mu_0|\ll 1$, we have the approximation of $\phi_0(\mu_0,\epsilon)$ in \re{CH:sol}. 
It yields from combining \re{CH:sol} and \re{CH:H} that
\begin{equation}\label{CH:H1}
    \mathcal{H}_n(\mu_0,\epsilon) = \left(\frac{\pi}{2}+n\pi\right)^2 - \frac{1}{\epsilon^2} + 3\mu_0^2 \epsilon^2 + O(\mu_0^3\epsilon^4).
\end{equation}
Differentiating with respect to $\mu_0^2$, we get
\begin{equation*}
    \frac{\p \mathcal{H}_n}{\p (\mu_0^2)}(0,\epsilon) = 3\epsilon^2 \neq 0,
\end{equation*}
when $\epsilon\neq 0$. Therefore, it follows from Implicit Function Theorem that, for each nonzero $\mu_0$ with $|\mu_0|\ll 1$, there exists a unique solution to $\mathcal{H}_n(\mu_0,\epsilon)=0$; more specifically, the solution is $\epsilon_n=\epsilon_n^0 + O(\mu_0^2)$, which is close to $\epsilon_n^0$.

Similar to the proof of Theroem \ref{bif1}, we can apply Crandall-Rabinowitz Theorem to derive the following theorems for the system \re{1dCH}:

\begin{theorem}\label{bifCH1}
For a fixed $\mu_0$ with $|\mu_0|\ll 1$, and each integer $n\ge 0$, let $\epsilon_n^{(1)} = \frac{1}{\pi/2 + n\pi}+ O(\mu_0^2)$. Then $\left(\phi_0(\mu_0,\epsilon_n^{(1)}),\epsilon_n^{(1)}\right)$ is a bifurcation point to the system \re{1dCH}, where $\phi_0(\mu_0,\epsilon_n^{(1)})$ denotes the middle solution of the equation \re{CH1}. In addition, the bifurcation solution $(\phi_n(x,s),\epsilon_n(s))$ can be represented by:
$$\epsilon_n(s)=\epsilon_n^{(1)}+s,\hspace{2em} \phi_n(x,s)=\phi_0(\mu_0,\epsilon_n^{(1)}) +  s\sin\left(\left(\frac{\pi}{2}+n\pi\right)x\right)+O(s^2),\hspace{2em}\text{where $|s|\ll1$.}$$
\end{theorem}

\begin{theorem}\label{bifCH2}
For a fixed $\mu_0$ with $|\mu_0|\ll 1$, and each integer $n\ge 1$, let $\epsilon_n^{(2)} = \frac{1}{n\pi} + O\left(\mu_0^2\right)$. Then $\left(\phi_0(\mu_0,\epsilon_n^{(2)}),\epsilon_n^{(2)}\right)$ is a bifurcation point to the system \re{1dCH}, where $\phi_0(\mu_0,\epsilon_n^{(2)})$ denotes the middle solution of the equation \re{CH1}. In addition, the bifurcation solution $(\phi_n(x,s),\epsilon_n(s))$ can be represented by:
$$\epsilon_n(s)=\epsilon_n^{(2)}+s,\hspace{2em} \phi_n(x,s)=\phi_0(\mu_0,\epsilon_n^{(2)}) +  s\cos(n\pi x)+O(s^2),\hspace{2em}\text{where $|s|\ll1$.}$$
\end{theorem}

\section{Binary system with long-range interaction}

In this section, we will apply the bifurcation analysis to some block copolymer system. Performing bifurcation analysis on such a system is challenging due to the nonlocal feature introduced by the long-range interaction terms. We will begin with a brief introduction on the background of block copolymer systems.

 Block copolymers are chain molecules made by several different segment species. It is called diblock copolymer system when it contains two species, say $A$ and $B$; it is called triblock copolymer system when it is made by three species, say $A$, $B$ and $C$. Generally we can have a block copolymer system with $N$ species. Due to the chemical incompatibility, different species tend to be phase-seperated. However, different species are connected by covalent chemical bonds, leading to the microphase separation.
 
 Block copolymers provide simple and easily controlled materials for the study of self-assembly. Mean field theory, which is associated with an free energy functional, have proven practically useful in understanding and predicting pattern morphology \cite{Bates_PhysToday1999, Choksi_NonlinearScience2001}. Ohta and Kawasaki proposed a model in \cite{OhtaKawasaki_Macromolecules1986}, which now we refer to Ohta-Kawasaki(OK) model, to study the phase separation of diblock copolymers. Several years later, Nakazawa and Ohta further proposed the Nakazawa-Ohta (NO) model \cite{OhtaNakazawa_Macromolecules1993} to explore the pattern formation of triblock copolymer system. In general, a mean field model of a block copolymer system with $N+1$ species $\{A_i\}_{i=1}^{N+1}$ with long-range interactions can be formulated as \cite{ChoiZhao_DCDSB2021}:
 \begin{align}\label{Energy_Ncomp}
	\begin{split}
		& E^N[\phi_1, \cdots, \phi_N] \\
		= & \int_{\Omega} \dfrac{\epsilon}{2} \sum_{\substack{i,j=1 \\ i \leq j}}^{N} \nabla \phi_i \cdot \nabla \phi_j + \dfrac{1}{2\epsilon} \left[ \sum_{i=1}^N W(\phi_i) + W \left( 1 - \sum_{i=1}^N \phi_i \right) \right] \dif x \\
		& + \sum_{i,j = 1}^N \frac{\gamma_{ij}}{2}  \int_{\Omega} \left[ \left(-\triangle\right)^{-\frac{1}{2}}  \left( \phi_{i} -\omega_{i}\right) \left(-\triangle\right)^{-\frac{1}{2}}\left( \phi_{j} -\omega_{j}\right) \right] \dif x.
	\end{split}
\end{align}
Here $0 < \epsilon \ll 1$ is an interface parameter, $\Omega = [-1,1]^d \subset \mathbb{R}^d, d=1, 2,3$, is the spatial domain. The function 
\begin{align}\label{Wfunc}
    W(\phi) = 18(\phi^2 - \phi)^2
\end{align}
is a double-well potential with two local minima at $0$ and $1$. Taking the sum of $W(\cdot)$ for all species, we introduce a potential function $W_N$:
\begin{align}\label{eqn:W_N}
	W_N(\phi_1, \cdots, \phi_N) = \frac{1}{2} \left[ \sum_{i=1}^N W(\phi_i) + W \left( 1 - \sum_{i=1}^N \phi_i \right) \right],
\end{align}
which has $N$ local minima at $(1,0,\cdots,0), (0,1,0,\cdots,0), \cdots, (0,\cdots,0,1)$. The first integral in (\ref{Energy_Ncomp}) accounts for the interfacial free energy between different species which is oscillation-inhibiting, therefore it favors large domains with minimal common area between species. The parameter $\gamma_{ij}$, assuming to be symmetric $\gamma_{ij} = \gamma_{ji}$, represents the strength of the long-range repulsive interaction between the $i$- and $j$-th species.  The parameter $\omega_i$, defined as
\begin{align*}
	\omega_i = \frac{1}{|\Omega|} \int_{\Omega} \phi_i \dif x, \quad  i = 1, \cdots, N
\end{align*}
represents the volume constraint for each species $\phi_i$.  The second term in (\ref{Energy_Ncomp}) is oscillation-forcing and therefore favors micro-domains of smaller size.

Note that when taking $N=1$ (respectively, $N=2$), the system (\ref{Energy_Ncomp}) reduces to the binary OK system (respectively, ternary NO system). In this work, we will focus on OK system, the binary system with a long-range interaction induced by $(-\Delta)^{-1}$:
\begin{align} \label{eqn:OK_energy}
E^{\text{OK}}[\phi] = \int_\Omega \left[ \frac{\epsilon}{2} | \nabla \phi|^2 + \frac{1}{\epsilon}W(\phi) \right] \dif x + \frac{\gamma}{2} \int_\Omega \left| (-\Delta)^{-\frac{1}{2}} \left( \phi - \omega \right) \right|^2 \dif x,	
\end{align}
with a volume constraint: 
\begin{align} \label{eqn:OK_VolConstraint}
 	 \int_\Omega \phi \dif x = \omega |\Omega|.
\end{align}
Here $\gamma$ measures the strength of the long-range repulsive force between bubbles. Fig \ref{figure:GammaEffect} shows the schematic of the $\gamma$ effect on the pattern of bubble assembly. The larger the value of $\gamma$ is, the more bubbles are formed at equilibrium.

In order to handle the volume constraint (\ref{eqn:OK_VolConstraint}),
 we introduce a penalty term to change (\ref{eqn:OK_energy}) into an unconstrained free energy functional:
\begin{align}\label{eqn:pOK_energy}
	\begin{split}
		E^{\text{pOK}}[\phi] =& \int_\Omega \left[ \frac{\epsilon}{2} | \nabla \phi|^2 + \frac{1}{\epsilon} W(\phi) \right] \dif x + \frac{\gamma}{2} \int_\Omega \left| (-\Delta)^{-\frac{1}{2}} \left( \phi - \omega \right) \right|^2 \dif x \\
		& + \frac{M}{2} \left( \int_\Omega \phi \dif x - \omega |\Omega| \right)^2,
	\end{split}
\end{align}
with $M\gg 1$ being the penalty constant for the volume constraint. Since we are interested in the pattern formation at equilibrium for OK system, the $L^2$ gradient flow dynamics for (\ref{eqn:pOK_energy}) is considered, leading us to the penalized Allen-Cahn-Ohta-Kawasaki (pACOK) equation for the time evolution of $\phi(x,t)$:
\begin{align}\label{eqn:pACOKss}
	\begin{split}
		\frac{\partial \phi}{\partial t} = - \frac{\delta E^{\text{pOK}} [\phi]}{\delta \phi} =& \epsilon \Delta \phi - \dfrac{1}{\epsilon} W'(\phi) - \gamma (-\Delta)^{-1} ( \phi - \omega ) \\
		& - M \left( \int_{\Omega} \phi \dif x - \omega |\Omega| \right),
	\end{split}
\end{align}
with a given initial condition $\phi(x,t=0) = \phi_0(x)$, and no flux boundary condition on $\partial\Omega$. The pACOK dynamics (\ref{eqn:pACOK}) satisfies the energy dissipative law
\begin{align*}
	\frac{d}{dt} E^{\text{pOK}}[\phi] = - \| \phi_t \|_{L^2}^2 \leq 0.
\end{align*}

There has been extensive work in the past decades to study OK model from both theoretical and numerical perspectives. In \cite{RenTruskinovsky_Elasticity2000,RenWei_SIAM2000}, the characterization of 1d global minimizers for OK system is established, and similar analysis is also applied to the system with $(I-\gamma^2\Delta)^{-1}$ type long-range interaction. Choski in \cite{Choksi_QAM2012} conducted asymptotic analysis for the global minimizers of OK model. Some variants of OK model are recently developed. For instance, Chan, Nejad and Wei in \cite{ChanNejadWei_PhysicaD2019} replaced the standard diffusion in OK model by a fractional diffusion, and the $\Gamma$-convergence and the existence of the global minimizers are proved. Recently there are also some work on the $L^2$ and $H^{-1}$ gradient flow dynamics for OK model. Joo, Xu and Zhao in \cite{JooXuZhao_IFB2021} studied the global well-posedness of the $L^2$ gradient flow dynamics for both OK and NO systems using the De Giorgi's minimizing movement scheme in which the volume constraints are handled via either Lagrange multiplier or penalty method. In \cite{WangRenZhao_CMS2019}, the authors provided numerical evidences that pACOK dynamics displays hexagonal bubble assembly. More importantly, as the first time, they numerically found that NO model forms a hexagonal double-bubble pattern at equilibrium when solving the penalized $L^2$ gradient flow for NO model. Recently efforts have been made to design numerical schemes for the $L^2$ gradient flow dynamics of OK model, such as first order operator-splitting energy stable methods \cite{XuZhao_JSC2019} and maximum principle preserving methods \cite{XuZhao_JSC2020}. Furthermore in \cite{ChoiZhao_DCDSB2021}, higher order energy stable schemes are designed to simulate the penalized $L^2$ gradient flow for OK/NO models. Additionally, $H^{-1}$ gradient flow dynamics is also considered recently. For instance, \cite{Benesova_SINA2014} studied an implicit midpoint spectral approximation for the equilibrium of OK model. \cite{ChengYangShen_JCP2017} adopts the IEQ method to study the diblock copolymer model.  Plus, theoretical studies of NO model is also attracting much attention in the past years. Ren and Wei studied a family of local minimizers with lamellar structure for the NO system in \cite{RenWei_JNS2003}, then they investigated the double bubble patterns of the triblock copolymer system in \cite{RenWei_ARMA2013,RenWei_ARMA2015}. However, the characterization of the (global) minimizers of NO system is still unclear. The global minimizers were only found in one-dimensional case for a degenerate case in \cite{GennipPeletier_CVPDE2008} in which the long-range interaction parameters $[\gamma_{ij}]$ was of some special form. A first attempt towards a systematic characterization of the global minimizers in non-degenerate case was conducted recently in \cite{XuDu_JNA2021}.

\begin{figure}
  \includegraphics[width=3in]{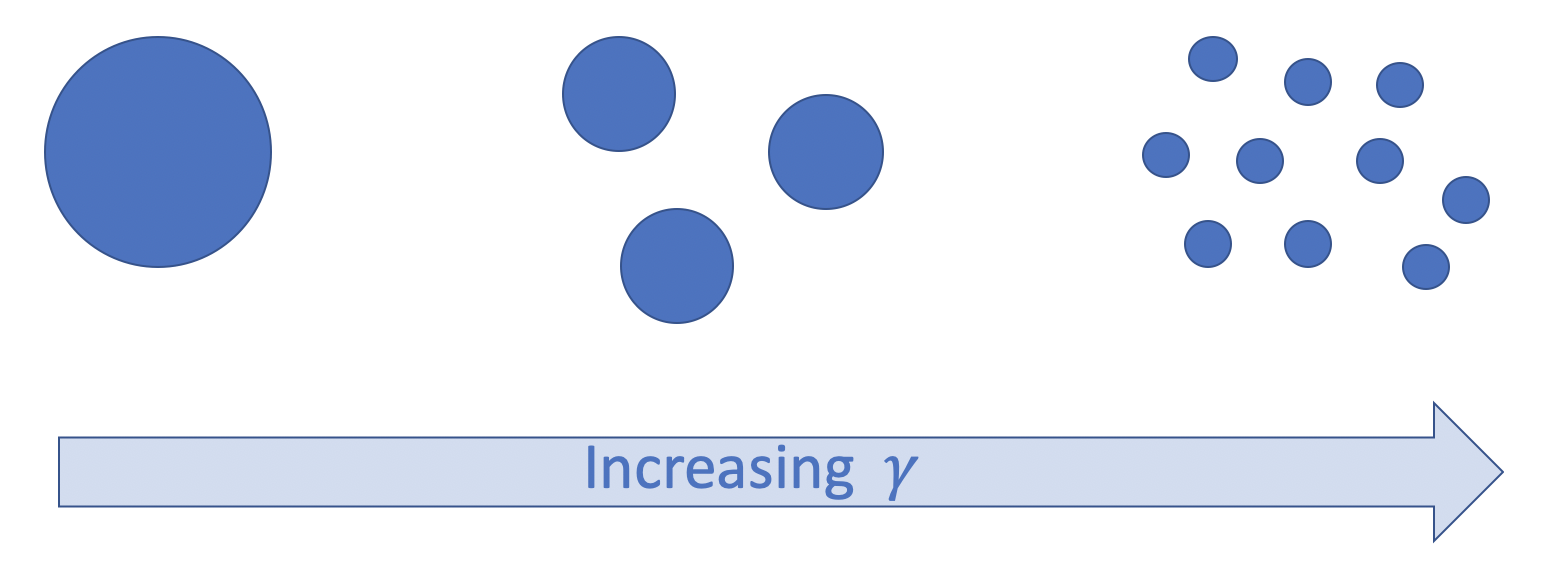}
  \caption{The schematic of $\gamma$-effect on solution patterns. Given a fixed relative volume $\omega$, the larger the long-range repulsive strength $\gamma$ is, the more bubbles the OK system displays at equalibrium.}\label{figure:GammaEffect}
\end{figure}

\subsection{Steady-state system and bifurcation analysis}

To perform bifurcation analysis to pACOK system (\ref{eqn:pACOKss}), some modification is needed. More specifically, we will add a term $(-\Delta)^{-1}(\phi - \bar{\phi})$ in the equation in replace of the terms involving the prefactor $\omega$. Then we consider the following steady-state system of (\ref{eqn:pACOKss}):
\begin{equation}\label{eqn4.1}
    \left\{
    \begin{aligned}
        &\epsilon \Delta \phi - \frac{1}{\epsilon} W'(\phi) - \gamma (-\Delta)^{-1}\left(\phi - \frac{1}{|\Omega|}\int_\Omega \phi \dif x\right) = 0 \hspace{2em} &&\text{in } \Omega,\\
        &\frac{\partial \phi}{\partial \bm{n}} = 0 \hspace{2em} &&\text{on } \partial \Omega.
    \end{aligned}
    \right.
\end{equation}
To solve this system, we introduce $u$ as 
\begin{equation}
    \label{eqn:u0}
    u=(-\Delta)^{-1}\left(\phi - \frac{1}{|\Omega|}\int_\Omega \phi \dif x\right),
\end{equation}
and impose the same Nenumann boundary condition for $u$. Therefore, the system for $u$ is:
\begin{equation}
    \label{eqn:u}
    \left\{
    \begin{aligned}
    &\Delta u = \frac{1}{|\Omega|}\int_\Omega \phi \dif x - \phi \triangleq f(x) \hspace{2em} &&\text{in }\Omega,\\
    &\frac{\partial u}{\partial \bm{n}} = 0 \hspace{2em} &&\text{on } \partial \Omega.
    \end{aligned}
    \right.
\end{equation}
We can solve $u$ from \re{eqn:u} by using the fundamental solution $G(x,y)$ of $\Delta u = 0$. To this end, we apply Green's second identity
\begin{equation}\label{eqn:u1}
\begin{split}
    u(x) &= \int_\Omega G(x,y)f(y)\dif y + \int_{\partial \Omega}\Big[ u(y)\frac{\p G(x,y)}{\p\bm{n}_y} - G(x,y)\frac{\partial u(y)}{\p \bm{n}_y}\Big]\dif S_y\\
    &= \int_\Omega G(x,y)f(y)\dif y + \int_{\partial \Omega} u(y)\frac{\p G(x,y)}{\p\bm{n}_y}\dif S_y,
    \end{split}
\end{equation}
In 1D, $G(x,y)=\frac12 |x-y|$. In addition, if we take $\Omega =[-1,1]$, then $\p\Omega = \{-1,1\}$, and $$\frac{\p G(x,y)}{\p\bm{n}_y}\bigg|_{y=1} = \frac{\p(\frac12(y-x))}{\p y}\cdot 1 = \frac12,\hspace{2em} \frac{\p G(x,y)}{\p\bm{n}_y}\bigg|_{y=-1} = \frac{\p(\frac12(x-y))}{\p y}\cdot (-1) = \frac12.$$
Substituting this result back into \re{eqn:u1}, we have
\begin{equation}
    \label{eqn:u2}
    u(x) = \int_\Omega G(x,y)f(y)\dif y + \frac12 \int_{\p \Omega} u(y)\dif S_y.
\end{equation}
The solution $u$ to the system \re{eqn:u} is unique up to a constant addition. We assume $u$ is of zero mean, $\int_\Omega u(x)\dif x=0$. Then
\begin{equation}
    \label{eqn:u3}
    \int_{\p \Omega} u(y)\dif S_y = -\frac{2}{|\Omega|} \int_\Omega \int_\Omega G(x,y)f(y)\dif y \dif x.
\end{equation}
Together with \re{eqn:u1} and \re{eqn:u2}, the solution $u$ to the system \re{eqn:u} with zero mean is
\begin{equation}\label{eqn:u4}
\begin{split}
    u(x) &= \int_\Omega G(x,y)f(y)\dif y - \frac{2}{|\Omega|}\int_\Omega \int_\Omega G(x,y)f(y)\dif y \dif x\\
    &= \int_\Omega G(x,y)\Big(\frac{1}{|\Omega|}\int_\Omega \phi \dif x - \phi(y)\Big)\dif y - \frac{2}{|\Omega|}\int_\Omega \int_\Omega G(x,y)\Big(\int_\Omega \phi \dif x - \phi(y)\Big)\dif y \dif x\\
    &= \Big(\frac{1}{|\Omega|}\int_\Omega \phi \dif x\Big) H(x)
    + \int_\Omega G(x,y)\phi(y)\dif y - \frac{2}{|\Omega|}\int_\Omega \int_\Omega G(x,y)\phi(y)\dif y \dif x,
    \end{split}
\end{equation}
where $H(x) = \int_\Omega G(x,y)\dif y - \frac{2}{|\Omega|}\int_\Omega \int_\Omega G(x,y)\dif y\dif x$. 

Substituting \re{eqn:u0} and \re{eqn:u4} into the system \re{eqn4.1}, it yields an equivalent system:
\begin{equation}\label{eqn:pACOK1}
\left\{
\begin{aligned}
    &\epsilon\Delta \phi - \frac{1}{\epsilon}W'(\phi) - \gamma\Big(\frac{1}{|\Omega|}\int_\Omega\phi \dif x\Big)H(x) + \gamma\int_\Omega G(x,y)\phi(y)\dif y && \\
    &\hspace{12em}- \frac{2\gamma}{|\Omega|}\int_\Omega\int_\Omega G(x,y)\phi(y)\dif y\dif x = 0\hspace{2em} &&\text{in }\Omega,\\
    &\frac{\p \phi}{\p \bm{n}} = 0 &&\text{on }\p \Omega,
\end{aligned}
	\right.
\end{equation}
We will analyze the bifurcation solutions based on the above system. 

We first consider the trivial steady-state solutions. If $\phi\equiv C$, where $C$ is a constant, then $f(x)=0$ in \re{eqn:u}, and hence $u(x)=0$. Therefore, the trivial steady-state solutions to \re{eqn:pACOK1} (or equivalently \re{eqn4.1}) are:
\begin{align}\label{eqn:pACOK}
	\begin{split}
		W'(\phi_0)  =&0\Rightarrow \phi_0=1, \phi_0=\frac{1}{2},\hbox{~and~}\phi_0=0.  
	\end{split}
\end{align}
Similar as in \re{eqn:shiftedSteadyAllenCahn}, we first need to shift the solution by $-\phi_0$ ($\phi_0 =0,1,\frac12$) in order to apply Crandall-Rabinowitz Theorem (Theorem \ref{CRthm}) on any trivial steady states. As a result, we obtain the following system:
\begin{equation}\label{eqn:pACOK2}
\left\{
\begin{aligned}
    &\epsilon\Delta \phi - \frac{1}{\epsilon}W'(\phi+\phi_0) - \gamma\Big(\frac{1}{|\Omega|}\int_\Omega\phi \dif x\Big)H(x) + \gamma\int_\Omega G(x,y)\phi(y)\dif y && \\
    &\hspace{12em}- \frac{2\gamma}{|\Omega|}\int_\Omega\int_\Omega G(x,y)\phi(y)\dif y\dif x = 0\hspace{2em} &&\text{in }\Omega,\\
    &\frac{\p \phi}{\p \bm{n}} = 0 &&\text{on }\p \Omega,
\end{aligned}
	\right.
\end{equation}
Notice that in deriving the above system, we made use of the fact that $u(x)=0$ if $\phi=\phi_0$ is a constant, hence,
\begin{equation}\label{equal}
    \Big(\frac{1}{|\Omega|}\int_\Omega \phi_0\dif x\Big)H(x) + \int_\Omega G(x,y)\phi_0 \dif y - \frac{2}{|\Omega|}\int_\Omega\int_\Omega G(x,y)\phi_0 \dif y \dif x = 0.
\end{equation}
After being shifted, $\phi=0$ is always a solution to the system \re{eqn:pACOK2} when $\phi_0=0,1,$ or $\frac12$.

Next, we shall use Crandall-Rabinowitz Theorem to find bifurcation points on the solution branches of constant steady state solutions. To this end, we define the following operator:
\begin{equation}
\begin{split}
    \label{ACOK-G}
    \mathcal{G}(\phi,\gamma) =&\; \epsilon \Delta \phi - \frac{1}{\epsilon} W'(\phi+\phi_0) - \frac{\gamma H(x)}{|\Omega|}\int_\Omega \phi \dif x + \gamma \int_\Omega G(x,y)\phi(y)\dif y \\
    &- \frac{2\gamma}{|\Omega|}\int_\Omega\int_\Omega G(x,y)\phi(y)\dif y\dif x.
    \end{split}
\end{equation}
In this case, $\gamma$ is viewed as the bifurcation parameter. Since $\mathcal{G}$ involves at most second-order derivatives of $\phi$, $\mathcal{G}(\cdot,\gamma)$ maps $X$ into $Y$, where the spaces $X$ and $Y$ are defined in \re{spaces}. As in Lemma \ref{lem1.1}, we need to compute the Fr\'echet derivative of $\mathcal{G}$.

\begin{lem}\label{lem3.1}
The Fr\'echet derivative $D_\phi \mathcal{G}(\phi,\gamma)$ of the operator $\mathcal{G}$ is given by
\begin{equation}
    \label{ACOK:freD}
    \begin{split}
   D_\phi \mathcal{G}(\phi,\gamma)[\xi] =&\; \epsilon \Delta \xi - \frac{1}{\epsilon} W''(\phi+\phi_0)\xi - \frac{\gamma H(x)}{|\Omega|}\int_\Omega \xi \dif x + \gamma \int_\Omega G(x,y)\xi(y)\dif y \\
    &- \frac{2\gamma}{|\Omega|}\int_\Omega\int_\Omega G(x,y)\xi(y)\dif y\dif x.
    \end{split}
\end{equation}
\end{lem}
\begin{proof}
Let $\phi,\xi\in X$. Since the operator $\mathcal{G}_1$ defined as follows
\begin{equation*}
        \mathcal{G}_1(\phi,\gamma) \triangleq\; \epsilon \Delta \xi  - \frac{\gamma H(x)}{|\Omega|}\int_\Omega \phi \dif x + \gamma \int_\Omega G(x,y)\phi(y)\dif y - \frac{2\gamma}{|\Omega|}\int_\Omega\int_\Omega G(x,y)\phi(y)\dif y\dif x,
\end{equation*}
is a linear operator with respect to $\phi$, we have $\mathcal{G}_1(\phi+\xi,\gamma) = \mathcal{G}_1(\phi,\gamma)+\mathcal{G}_1(\xi,\gamma)$. It follows from the Mean Value Theorem that
\begin{equation*}
\begin{split}
    \mathcal{G}(\phi+\xi,\gamma) - \mathcal{G}(\phi,\gamma)  &= \Big[\mathcal{G}_1(\phi+\xi,\gamma) - \frac{1}{\epsilon}W'(\phi+\xi+\phi_0)\Big] -\Big[ \mathcal{G}_1(\phi,\gamma) - \frac{1}{\epsilon}W'(\phi+\phi_0)\Big] \\
    &= \mathcal{G}_1(\xi,\gamma) - \frac{1}{\epsilon}\Big[W'(\phi+\xi+\phi_0) - W'(\phi+\phi_0) \Big]\\
    &= \mathcal{G}_1(\xi,\gamma) - \frac{1}{\epsilon}\Big[W''(\phi+\phi_0)\xi + W'''(\psi)\xi^2 \Big],
    \end{split}
\end{equation*}
where $\psi \in (\phi+\phi_0,\phi+\xi+\phi_0)$. Combining with \re{ACOK:freD}, we have
\begin{equation*}
    \mathcal{G}(\phi+\xi,\gamma) - \mathcal{G}(\phi,\gamma) - D_\phi \mathcal{G}(\phi,\gamma)[\xi] = -\frac{1}{\epsilon}W'''(\psi)\xi^2.
\end{equation*}
Therefore, we find that
\begin{equation*}
    \frac{\|\mathcal{G}(\phi+\xi,\gamma) - \mathcal{G}(\phi,\gamma) - D_\phi \mathcal{G}(\phi,\gamma)[\xi]\|_Y}{\|\xi\|_X} = \frac{1}{\epsilon}\frac{\|W'''(\psi)\xi^2\|_Y}{\|\xi\|_X} \le \frac{C}{\epsilon}\frac{\|\xi\|_X^2}{\|\xi\|_X} \rightarrow 0,
\end{equation*}
as $\|\xi\|_X\rightarrow 0$, which leads to the Fr\'echet derivative in \re{ACOK:freD}.
\end{proof}

Based on Lemma \ref{lem3.1} and equation \re{ACOK:freD}, we have
\begin{equation}
    \label{ACOK:freD0}
    \begin{split}
    D_\phi \mathcal{G}(0,\gamma)[\xi] =&\; \epsilon \Delta \xi - \frac{1}{\epsilon} W''(\phi_0)\xi - \frac{\gamma H(x)}{|\Omega|}\int_\Omega \xi \dif x + \gamma \int_\Omega G(x,y)\xi(y)\dif y \\
    &- \frac{2\gamma}{|\Omega|}\int_\Omega\int_\Omega G(x,y)\xi(y)\dif y\dif x,
    \end{split}
\end{equation}
where $W''(\phi_0) = 36(6\phi_0^2-6\phi_0+1)$. In what follows, we shall consider the bifurcations according to different values of $\phi_0$. 

\subsubsection{Bifurcations around  $\phi_0=\frac12$}
When $\phi_0=\frac12$, $W''(\phi_0)=-18$, it follows that,
\begin{equation}
    \label{case1}
    \begin{split}
    D_\phi \mathcal{G}(0,\gamma)[\xi] =&\; \epsilon \Delta \xi +\frac{18}{\epsilon}\xi - \frac{\gamma H(x)}{|\Omega|}\int_\Omega \xi \dif x + \gamma \int_\Omega G(x,y)\xi(y)\dif y \\
    &- \frac{2\gamma}{|\Omega|}\int_\Omega\int_\Omega G(x,y)\xi(y)\dif y\dif x.
    \end{split}
\end{equation}
Recall that $\Omega=[-1,1]$. Hence, if $\xi\in X$ is in the kernel of $D_\phi \mathcal{G}(0,\gamma)$, it is equivalent to
\begin{equation*}
    \begin{cases}
        &\epsilon\xi_{xx} + \frac{18}{\epsilon}\xi -\frac{\gamma H(x)}{2}\int_{-1}^1 \xi(x)\dif x + \gamma \int_{-1}^1 G(x,y)\xi(y)\dif y - \gamma\int_{-1}^1\int_{-1}^1 G(x,y)\xi(y)\dif y \dif x = 0,\\
        &\xi_x(-1) = \xi_x(1) = 0.
    \end{cases}
\end{equation*}
Like in the previous analysis, we once again use the ansatz \re{eigen} to derive two solution expressions, \re{sol1} and \re{sol2}. Substituting each term of \re{sol1} into \re{case1} leads to
\begin{align}\label{a0}
    D_{\phi} \mathcal{G}(0,\gamma)[a_0] = \frac{18}{\epsilon}a_0, 
\end{align}
where we use \re{equal} to cancel the last three terms; in addition, recalling that $G(x,y)=\frac12|x-y|$ in 1D, we have
\begin{equation*}
\begin{split}
    &D_\phi \mathcal{G}(0,\gamma)\left[b_n\sin\left(\left(\frac{\pi}{2}+n\pi\right)x \right)\right] =\\
    &\hspace{2em}\left[-\epsilon\left(\frac{\pi}{2}+n\pi\right)^2 +\frac{18}{\epsilon}\right] b_n \sin\left(\left(\frac{\pi}{2}+n\pi\right)x \right)- \frac{\gamma H(x)b_n}{2}\int_{-1}^1 \sin\left(\left(\frac{\pi}{2}+n\pi\right)x \right)\dif x\\
    &\hspace{2em}+ \gamma b_n\int_{-1}^1 \frac12 |x-y|\sin\left(\left(\frac{\pi}{2}+n\pi\right)y \right)\dif y -\gamma b_n\int_{-1}^1\int_{-1}^1 \frac12|x-y|\sin\left(\left(\frac{\pi}{2}+n\pi\right)y\right)\dif y \dif x,
    \end{split}
\end{equation*}
where
\begin{equation*}
    \begin{split}
        &\quad\ \int_{-1}^1   |x-y|\sin\left(\left(\frac{\pi}{2}+n\pi\right)y\right)\dif y\\
        &= \int_{-1}^x (x-y)\sin\left(\left(\frac{\pi}{2}+n\pi\right)y \right)\dif y + \int_{x}^1 (y-x)\sin\big((\frac{\pi}{2}+n\pi)y \big)\dif y\\
        &= -\frac{2}{(\frac{\pi}2+n\pi)^2}\sin\left(\left(\frac{\pi}{2}+n\pi\right)x \right) + \frac{2\cos\big(\frac{\pi}{2}+n\pi \big)}{\frac{\pi}2+n\pi}x\\
        &= -\frac{2}{(\frac{\pi}2+n\pi)^2}\sin\left(\left(\frac{\pi}{2}+n\pi\right)x \right) + 0  = -\frac{2}{(\frac{\pi}2+n\pi)^2}\sin\left(\left(\frac{\pi}{2}+n\pi\right)x \right).
    \end{split}
\end{equation*}
The last equality holds since $n$ is an integer, then $\cos\left(\frac{\pi}{2}+n\pi \right)=0$. Integrating one more time leads to
\begin{equation*}
\begin{split}
    \int_{-1}^1\int_{-1}^1 |x-y|\sin\left(\left(\frac{\pi}{2}+n\pi\right)y \right)\dif y \dif x&= -\frac{2}{(\frac{\pi}2+n\pi)^2}\int_{-1}^1 \sin\left(\left(\frac{\pi}{2}+n\pi\right)x \right) \dif x  = 0.
    \end{split}
\end{equation*}
Hence,
\begin{equation}
    \label{bn}
    D_\phi \mathcal{G}(0,\gamma)\left[b_n\sin\left(\left(\frac{\pi}{2}+n\pi\right)x \right)\right] = \Big[-\epsilon\left(\frac{\pi}{2}+n\pi\right)^2 +\frac{18}{\epsilon} - \frac{\gamma}{(\frac{\pi}2+n\pi)^2}\Big] b_n \sin\left(\left(\frac{\pi}{2}+n\pi\right)x \right).
\end{equation}
Combining the above equation with \re{a0}, we obtain
\begin{equation}\label{case1:eq1}
    D_\phi \mathcal{G}(0,\gamma)[\xi] = \frac{18 a_0}{\epsilon} + \sum_{n=0}^\infty\Big[-\epsilon\left(\frac{\pi}{2}+n\pi\right)^2 +\frac{18}{\epsilon} -\frac{\gamma}{(\frac{\pi}{2}+n\pi)^2}\Big]b_n\sin\left(\left(\frac{\pi}{2}+n\pi\right)x \right).
\end{equation}
If $-\epsilon(\frac{\pi}{2}+n\pi)^2 +\frac{18}{\epsilon} -\frac{\gamma}{(\frac{\pi}{2}+n\pi)^2}=0$, which is equivalent to $\gamma = -\epsilon(\frac{\pi}{2}+n\pi)^4+\frac{18}{\epsilon}(\frac{\pi}{2}+n\pi)^2$, then the term of $\sin\big((\frac{\pi}{2}+n\pi)x \big)$ disappears while all the other terms remain. Therefore, we have the following bifurcation theorem for the system \re{eqn:pACOK2}:

\begin{theorem}\label{ACOK:thm1}
For each integer $n\ge 0$, $\gamma^{(1)}_n = -\epsilon(\frac{\pi}{2}+n\pi)^4+\frac{18}{\epsilon}(\frac{\pi}{2}+n\pi)^2$, $\left(0,\gamma^{(1)}_n\right)$ is a bifurcation point to the system \re{eqn:pACOK2} such that there is a bifurcation solution $(\phi_n(x,s),\gamma_n(s))$ with
$$
\gamma_n(0)=\gamma_n^{(1)},\hspace{2em} \phi_n(x,s)=s\sin\left(\left(\frac{\pi}{2}+n\pi\right)x\right)+o(s),\hspace{2em}\text{where $|s|\ll1$.}
$$
\end{theorem}
\begin{proof}
We need to verify the four conditions in Crandall-Rabinowitz Theorem at the point $(0,\gamma^{(1)}_n)$. To begin with, the first condition is naturally satisfied, since $\phi=0$ is always a solution to \re{eqn:pACOK2}. When $\gamma = \gamma^{(1)}_n = -\epsilon(\frac{\pi}{2}+n\pi)^4+\frac{18}{\epsilon}(\frac{\pi}{2}+n\pi)^2$, it follows from \re{case1:eq1} that
\begin{equation*}
    \text{Im}\, D_\phi \mathcal{G}(0,\gamma_n^{(1)}) = \text{span}\Big\{1, \sin\big(\frac{\pi}{2}x\big), \sin\big(\frac{3\pi}{2}x\big),\cdots,\sin\big((\frac{\pi}2 +(n-1))x\big), \sin\big((\frac{\pi}2 +(n+1))x\big),\cdots\Big\},
\end{equation*}
and
\begin{equation*}
    \text{Ker}\,D_\phi \mathcal{G}(0,\gamma_n^{(1)}) = \text{span}\Big\{\sin\big((\frac{\pi}{2}+n\pi)x\big) \Big\},
\end{equation*}
which imply that  $\text{dim}(\text{Ker}\,D_\phi \mathcal{G}(0,\gamma_n^{(1)})) = 1$ and $\text{codim}(\text{Im}\, D_\phi \mathcal{G}(0,\gamma_n^{(1)}))=1$. Hence, the codimensional space and the non-tangential space meet the requirements of Crandall-Rabinowitz Theorem. To complete the proof, it remains to show the last condition. Differentiating \re{ACOK:freD0} with respect to $\gamma$ and applying on $\sin\big((\frac{\pi}{2}+n\pi)x\big)$, we have
\begin{equation}
    D_{\gamma \phi} \mathcal{G}(0,\gamma^{(1)}_n)\left[\sin\left(\left(\frac{\pi}{2}+n\pi\right)x\right)\right] = -\frac{1}{(\frac{\pi}2+n\pi)^2}\sin\left(\left(\frac{\pi}{2}+n\pi\right)x\right) \notin \text{Im}\, D_\phi \mathcal{G}(0,\gamma_n^{(1)}).
\end{equation}

All conditions of Crandall-Rabinowitz Theorem are satisfied, hence the results in Theorem \re{ACOK:thm1} are direct consequences of Theorem \ref{CRthm}.
\end{proof}

Since the solutions to the original model \re{eqn:pACOK} are shifted, we have the following bifurcation theorem for the original model:
\begin{theorem}\label{ACOK:thm2}
For each integer $n\ge 0$, $\gamma^{(1)}_n = -\epsilon(\frac{\pi}{2}+n\pi)^4+\frac{18}{\epsilon}(\frac{\pi}{2}+n\pi)^2$, $\left(\frac12,\gamma^{(1)}_n\right)$ is a bifurcation point to the system \re{eqn:pACOK} such that there is a bifurcation solution $(\phi_n(x,s),\gamma_n(s))$ with
$$
\gamma_n(0)=\gamma_n^{(1)},\hspace{2em} \phi_n(x,s)=\frac12+ s\sin\left(\left(\frac{\pi}{2}+n\pi\right)x\right)+o(s),\hspace{2em}\text{where $|s|\ll1$.}
$$
\end{theorem}

Similarly, considering solution \re{sol2}, we have
\begin{equation}\label{case1:eq2}
    D_\phi \mathcal{G}(0,\gamma)[\xi] = D_\phi \mathcal{G}(0,\gamma)\left[\sum_{n=0}^\infty a_n\cos(n\pi x)\right] = \frac{18 a_0}{\epsilon} + \sum_{n=1}^\infty\Big[-\epsilon(n\pi)^2 +\frac{18}{\epsilon} -\frac{\gamma}{(n\pi)^2}\Big]a_n\cos(n\pi x),
\end{equation}
then we obtain the bifurcation theorem for cosine modes:
\begin{theorem}\label{ACOK:thm3}
For each integer $n\ge 1$, $\gamma^{(2)}_n = -\epsilon(n\pi)^4+\frac{18}{\epsilon}(n\pi)^2$, $\left(\frac12,\gamma^{(2)}_n\right)$ is a bifurcation point to the system \re{eqn:pACOK} such that there is a bifurcation solution $(\phi_n(x,s),\gamma_n(s))$ with
$$\gamma_n(0)=\gamma_n^{(2)},\hspace{2em} \phi_n(x,s)=\frac12+ s\cos(n\pi x)+o(s),\hspace{2em}\text{where $|s|\ll1$.}$$
\end{theorem}

\subsubsection{No bifurcations around $\phi_0=0,1$}
When $\phi_0=0$ or $1$, we have $W''(\phi_0) = 36$, and \re{ACOK:freD0} becomes
\begin{equation*}
    \begin{split}
    D_\phi \mathcal{G}(0,\gamma)[\xi] =&\; \epsilon \Delta \xi -\frac{36}{\epsilon}\xi - \frac{\gamma H(x)}{|\Omega|}\int_\Omega \xi \dif x + \gamma \int_\Omega G(x,y)\xi(y)\dif y \\
    &- \frac{2\gamma}{|\Omega|}\int_\Omega\int_\Omega G(x,y)\xi(y)\dif y\dif x.
    \end{split}
\end{equation*}
Substituting solution \re{sol1}, it follows that
\begin{equation*}
    D_\phi \mathcal{G}(0,\gamma)[\xi] = -\frac{36 a_0}{\epsilon} + \sum_{n=0}^\infty\Big[-\epsilon(\frac{\pi}{2}+n\pi)^2 -\frac{36}{\epsilon} -\frac{\gamma}{(\frac{\pi}{2}+n\pi)^2}\Big]b_n\sin\left(\left(\frac{\pi}{2}+n\pi\right)x \right).
\end{equation*}
Notice that the coefficient for each $\sin\big((\frac{\pi}{2}+n\pi)x \big)$ term,  $\Big[-\epsilon(\frac{\pi}{2}+n\pi)^2 -\frac{36}{\epsilon} -\frac{\gamma}{(\frac{\pi}{2}+n\pi)^2}\Big]$, is always negative when $\epsilon,\gamma>0$, hence $\sin\big((\frac{\pi}{2}+n\pi)x \big)$ cannot be varnished. Similarly, if we use the solution expression \re{sol2}, the coefficient for each $\cos(n\pi x)$ term is also negative. Hence $\text{Ker}\,D_\phi \mathcal{G}(0,\gamma) = \emptyset$ for any $\gamma>0$, which does not meet the second condition of Crandall-Rabinowitz Theorem. Therefore, there are no bifurcation solutions around $\phi=0$ or $\phi=1$.

\subsection{Bifurcation diagram}
First, we compute the bifurcation points of $\gamma$ for different $\epsilon$ in Fig. \ref{figure:Gamma_epsilon}. It shows that more bifurcation points when $\epsilon$ is smaller. 
Next, we compute the bifurcation diagram of \re{eqn4.1} in Fig. \ref{figure:SS} with respect to $\gamma$ for $\epsilon=0.3$. It shows more complex structure than Allen-Cahn equation. For any given $\gamma$, we can also compute the multiple solutions based on the bifurcation diagram. For instance, when $\gamma=1000$ and $\epsilon=0.3$, we have 12 solutions shown in Fig. \ref{figure:sols}.

\begin{figure}
  \includegraphics[width=3.2in]{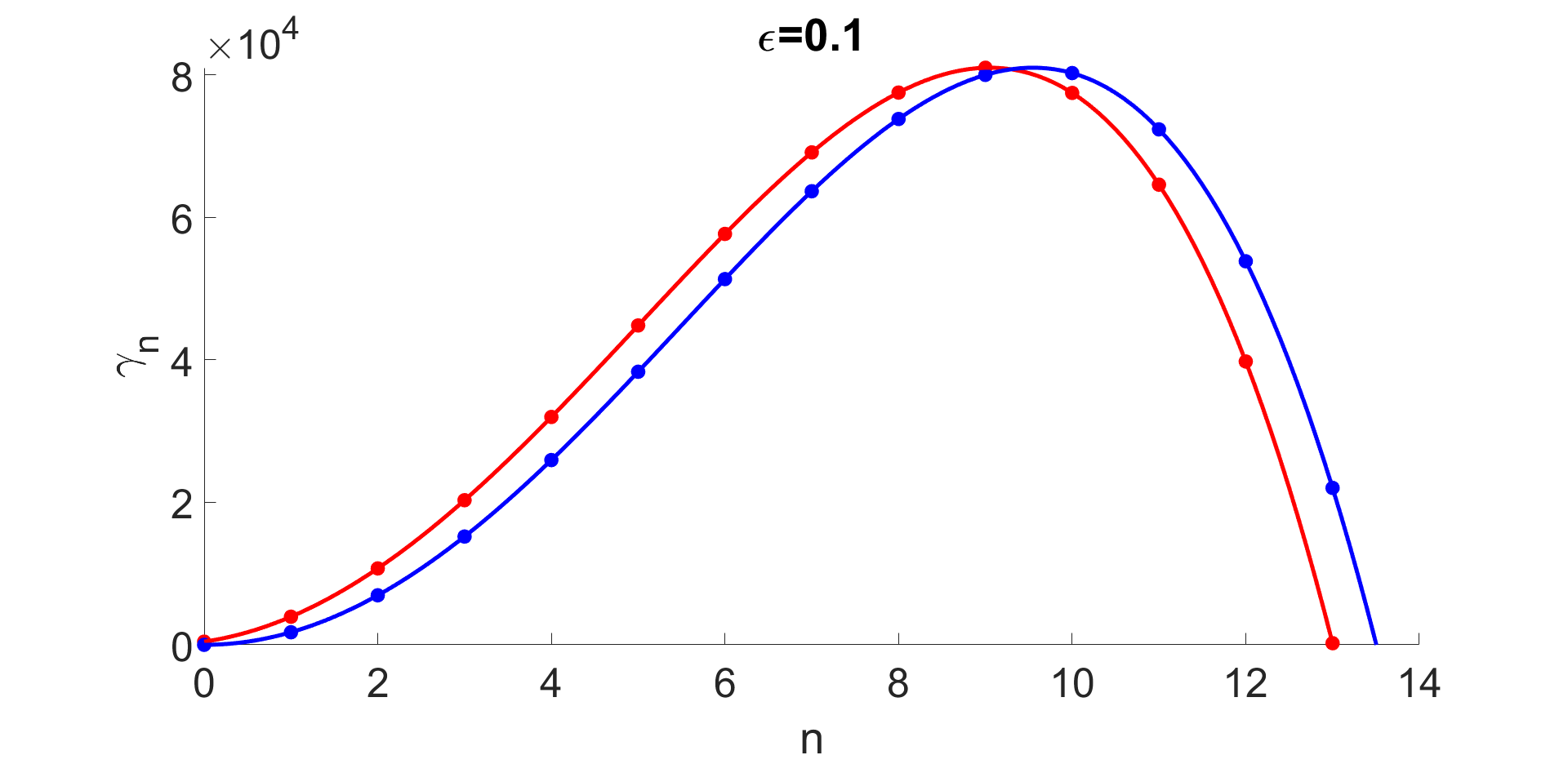}
    \includegraphics[width=3.2in]{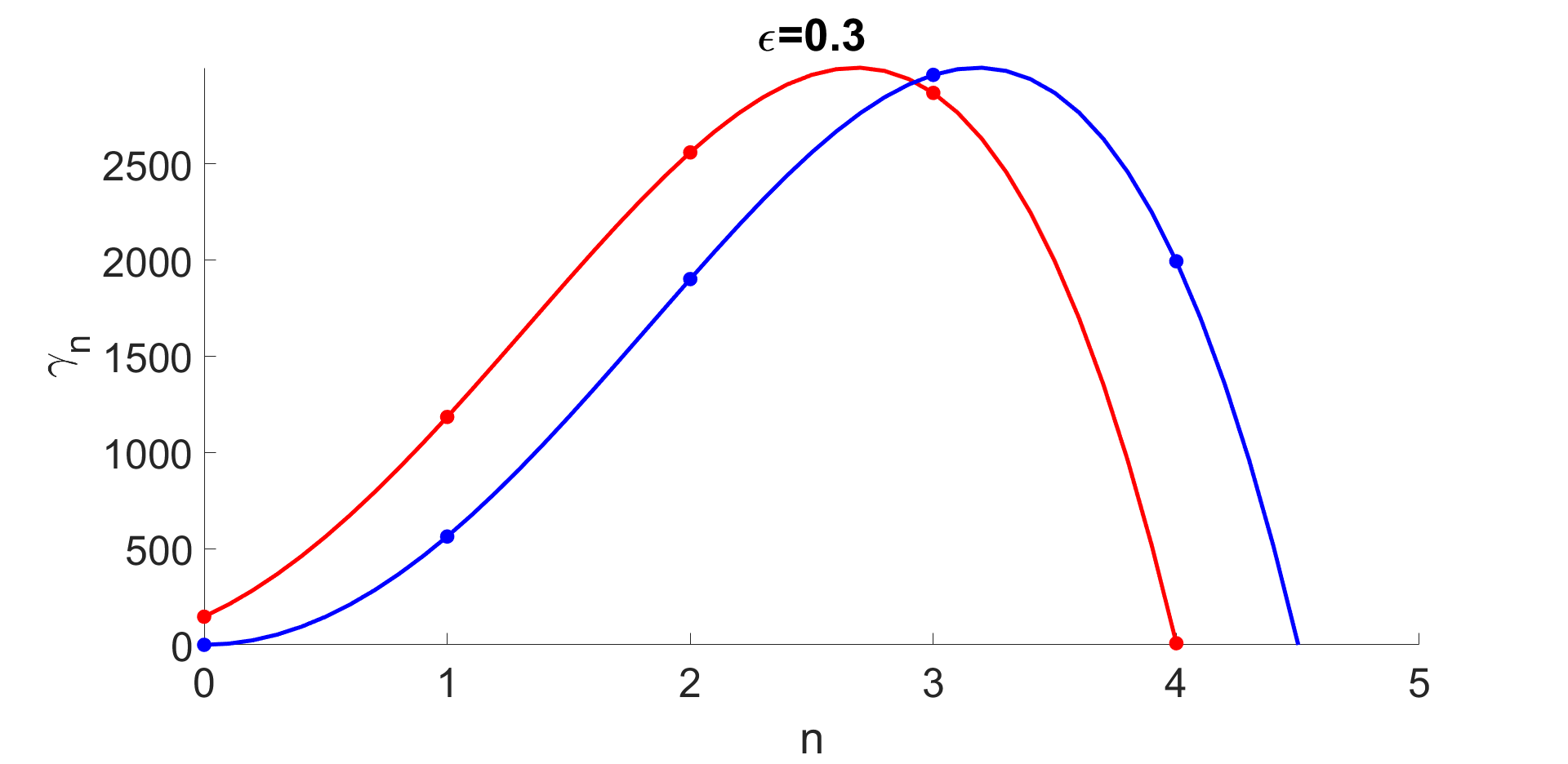}
  \caption{The bifurcation points $\gamma_n$ for different mode $n$ with $\epsilon=0.3$ (left) and $\epsilon=0.1$ (right). The red curve is corresponding to $\sin$ perturbation while the blue curve is for $\cos$ perturbation.}\label{figure:Gamma_epsilon}
\end{figure}

\begin{figure}
  \includegraphics[width=6in]{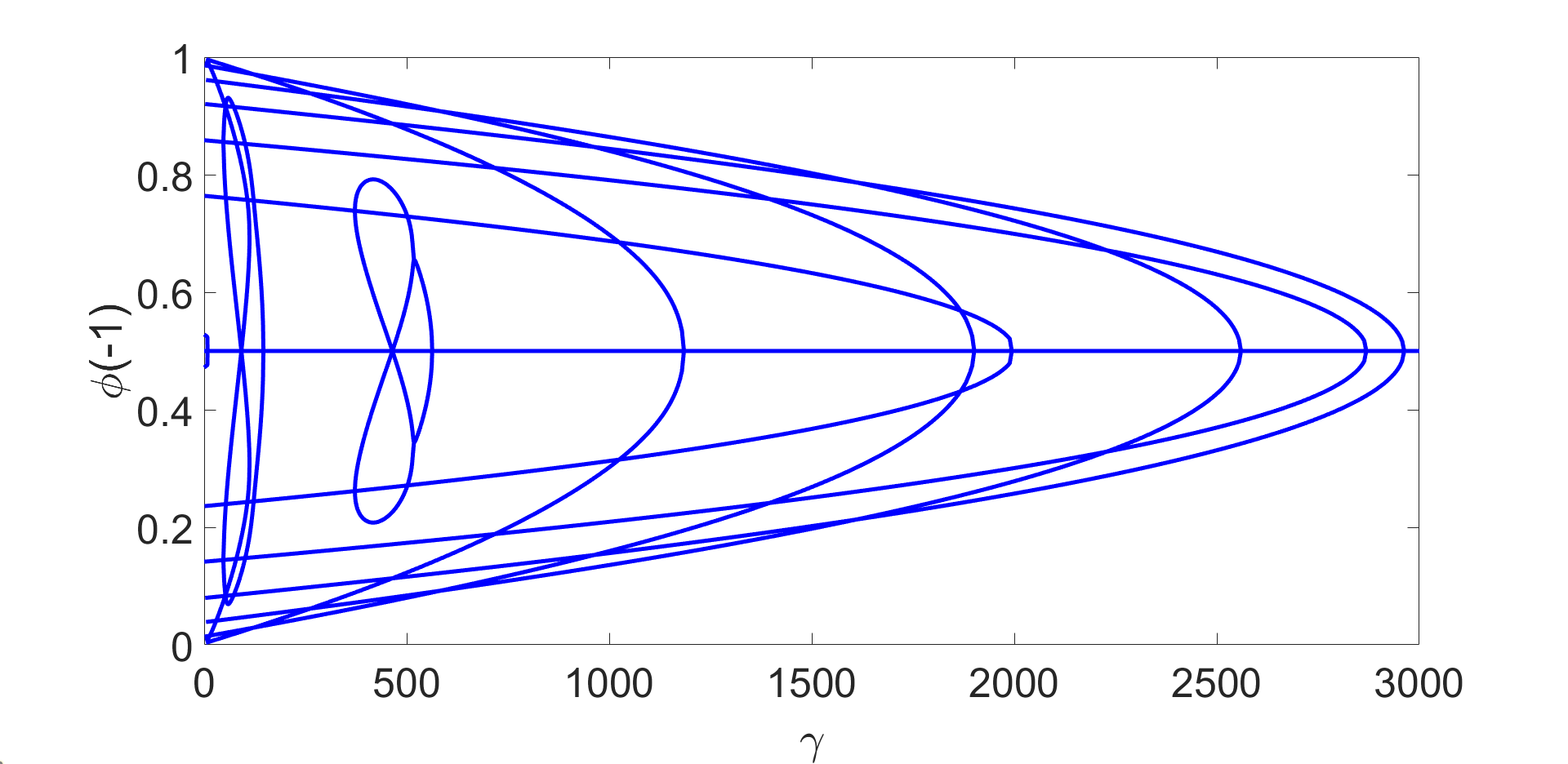}
  \caption{The solution structure of the ACOK model with respect to $\gamma$ with $\epsilon=0.3$. Here y-axis labels $\phi(-1)$. 
  $\epsilon=0.1$.}\label{figure:SS}
\end{figure}

\begin{figure}
  \includegraphics[width=6in]{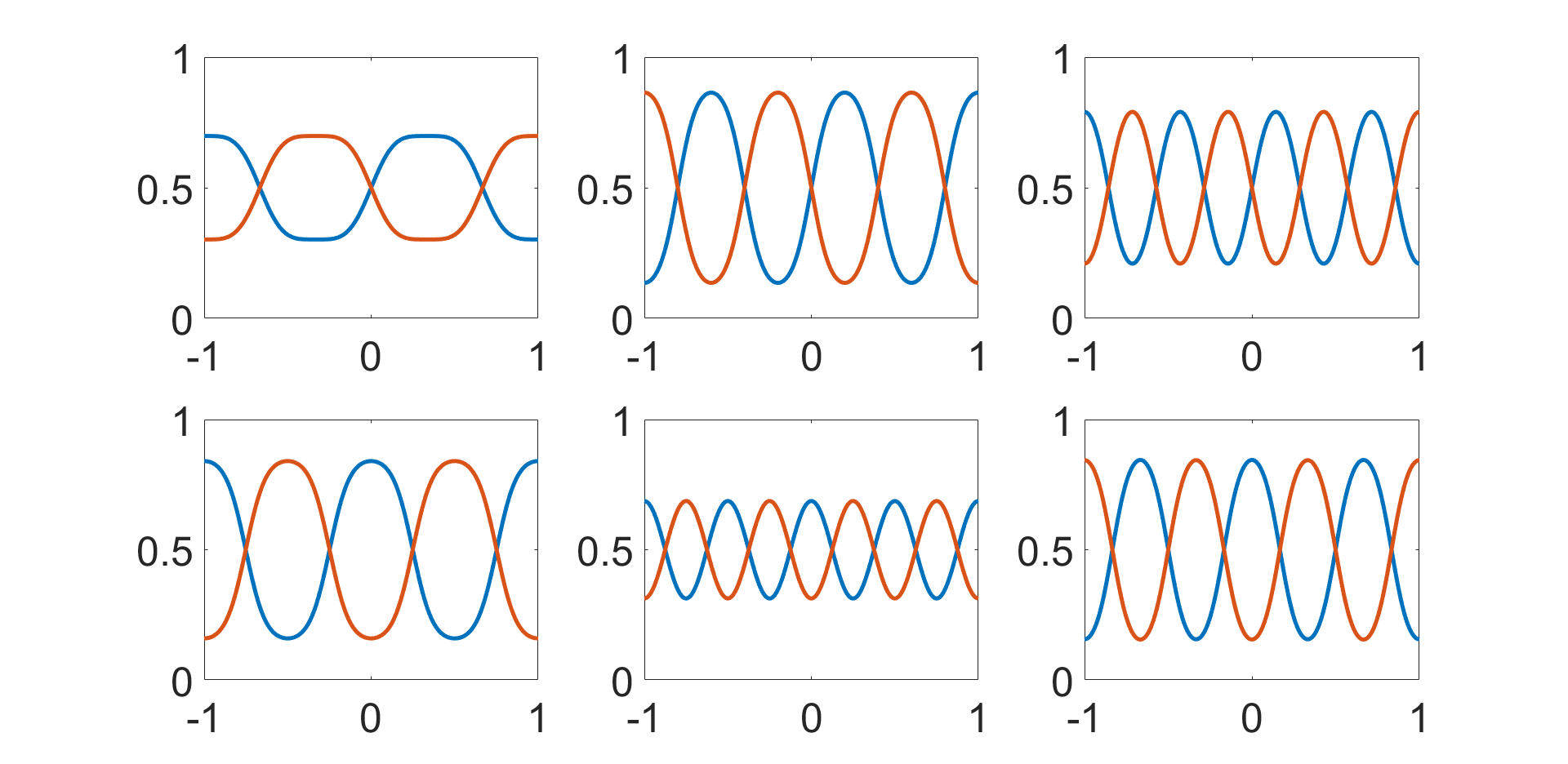}
  \caption{Multiple non-trivial solutions of ACOK model with $\gamma=1000$ and $\epsilon=0.3$. There are 12 solutions,  each panel consisting two symmetric solutions with respect to $y=\frac{1}{2}$.
  }\label{figure:sols}
\end{figure}

\section{Discussions and Conclusions}
We develop an analytical framework based on bifurcation analysis to explore the solution structure of phase field equations. It is applied to three well-known phase field equations, Allen-Cahn equation, Cahn-Hillard equation, and Allen-Chan-Ohta-Kawasaki system. Our results show that all the solutions bifurcate from the unstable trivial solution branch. Theoretical bifurcation analysis and numerical computation are presented to systematically validate the solution structures.

Note that the bifurcation analysis near $\phi_0 = 0$ leads to the solution structures for non-trivial equilibrium solutions. This analytical approach can be applied to other complex systems as long as they have a trivial solution or an analytical non-trivial solution. This has been successfully demonstrated in analyzing complex free boundary problems in tumor growth \cite{friedman2001symmetry,friedman2007bifurcation,hao2012bifurcation,zhao2020symmetry} and plaque formation \cite{hao2020bifurcation,zhao2021bifurcation}, but there are some limitations for current analytical approach. For example, it cannot be applied directly to cell migration on micro-patterns \cite{Zhao_PRE2017} in which cell is represented by a phase field function and winds up with a (time-dependent) periodic circular motion.

Besides, bifurcation analysis could have other impacts except providing the solution structure. For instance, we can use bifurcation results to obtain the stability condition of various numerical schemes. Consider the implicit scheme of Allen-Cahn equation as an example:
\begin{equation}
\frac{\phi^{n+1}-\phi^n}{\Delta t}-\Delta \phi^{n+1} + \frac{1}{\epsilon} ((\phi^{n+1})^3 - \phi^{n+1}) = 0, \hspace{2em} 0\leq i\leq N.
\end{equation}
Let $\phi^{n+1} = \phi^n +\psi$, then the linearized system reads 
\begin{equation}
    \frac{\psi}{\Delta t} - \psi_{xx} + \frac{1}{\epsilon^2}\big(3(\phi^n)^2 \psi -\psi \big) = \phi^n_{xx} - \frac{1}{\epsilon^2}\big((\phi^n)^3 - \phi^n \big)\label{eq2}.
\end{equation}
If we choose $\phi^n=0$, Eq. (\ref{eq2}) becomes
\begin{equation}
\frac{\psi}{\Delta t}-\psi_{xx}- \frac{1}{\varepsilon^2}\psi =0,
\end{equation}
which admits bifurcations, based on our results, when $\frac{1}{\Delta t}- \frac{1}{\epsilon^2} < 0$, i.e. $\Delta t > \epsilon^2$. Therefore, to ensure the uniqueness of $\phi^{n+1}$, the stability condition of the implicit scheme is  $\Delta t < \epsilon^2$. We will perform systematic study on the stability conditions for various numerical schemes in the future.

The bifurcation analysis we propose in this work can be naturally extended to the 2D case, in which the solutions are still branched out of the unstable constant solution $\phi = 0$. Only it has more complicated formulations and more time-consuming numerical computation than the 1D case. When applying to 2D case, an interesting and more complicated example of phase field models is the one with dynamic boundary conditions such as GMS model \cite{goldstein2011cahn} and LW model \cite{liu2019energetic,knopf2021phase}. For these models, the phase field function $\phi$ still satisfies the usual dynamics over the bulk domain $\Omega$ such as Cahn-Hilliard dynamics, coupled with standard boundary conditions such as Neumann boundary condition. However, $\phi|_{\partial\Omega}$ undergoes another phase field dynamics such as Allen-Cahn, or Cahn-Hilliard, or Allen-Cahn-Ohta-Kawasaki type. Our bifurcation framework is applicable to this type of model as $\phi = 0$ is still the solution branch containing bifurcation branches to other solutions. To this end, we need to design stable and efficient numerical discretization schemes for the phase field models with dynamic boundary conditions, though [P. Knoph et al.2021] has introduced some structure-preserving numerical schemes for solving GMS model and LW model. We will leave the bifurcation analysis together with stable and efficient solvers for phase field models with dynamic boundary conditions for the future consideration.

Current work for pACOK system is limited to the case in which phases or species $A$ and $B$ are of equal fraction. Some solution structures for a more general case of $\omega\ll 1$ can be explored based on the bifurcation from the nontrivial solution branch for the case equal fractions. What is more, this framework can be applied to analyze the solutions structures of other phase field equations, e.g.,  nonlocal Allen-Cahn equation, Allen-Cahn-Ohta-Kawazaki equation with a general nonlocal long-range interaction, and minimal phase field models for cell migration.


\section{Acknowledgement}
The work is primarily supported as part of the Computational Materials Sciences Program funded by the U.S. Department of Energy, Office of Science, Basic Energy Sciences, under Award No. DE-SC0020145. Y.Z. would like to acknowledge support for his effort by the Simons Foundation through Grant No. 357963 and NSF grant DMS-2142500.

\section{Conflict of interest}
On behalf of all authors, the corresponding author states that there is no conflict of interest.

\bibliography{ref}

\begin{thebibliography}{10}

\bibitem{Akrivis_MCAMS1998}
G.~Akrivis, M.~Crouzeix, and C.~Makridakis.
\newblock Implicit-explicit multistep finite element methods for nonlinear
  parabolic problems.
\newblock {\em Mathematics of Computation of the American Mathematical
  Society}, 67:457, 1998.

\bibitem{AndersonMcFaddenWheeler_ARFM1998}
D.~M. Anderson, G.~B. McFadden, and A.~A. Wheeler.
\newblock Diffuse-interface methods in fluid mechanics.
\newblock {\em Annu. Rev. Fluid Mech.}, 30:139, 1998.

\bibitem{Bates_PhysToday1999}
F.~S. Bats and G.~H. Fredrickson.
\newblock Block copolymers - designer soft materials.
\newblock {\em Phys. Today}, 52(2):32, 1999.

\bibitem{BKHR}
L.~Bauer, H.~B Keller, and E.~Reiss.
\newblock Multiple eigenvalues lead to secondary bifurcation.
\newblock {\em Siam Review}, 17(1):101--122, 1975.

\bibitem{Benesova_SINA2014}
B.~Benesova, C.~Melcher, and E.~Suli.
\newblock An implicit midpoint spectral approximation of nonlocal
  {C}ahn-{H}illiard equations.
\newblock {\em SIAM J. Numer. Anal.}, 52:1466, 2014.

\bibitem{Borden_CMAME2012}
M.~J. Borden, C.~V. Verhoosel, M.~A. Scott, T.~J. Hughes, and C.~M. Landis.
\newblock A phase-field description of dynamic brittle fracture.
\newblock {\em Comput. Methods Appl. Mech. Eng.}, 217:77, 2012.

\bibitem{BrowerKesslerKoplikLevine_PRA1984}
R.~Brower, D.~Kessler, J.~Koplik, and H.~Levine.
\newblock Geometrical models of interface evolution.
\newblock {\em Phy. Rev. A}, 29:1335, 1984.

\bibitem{AllenCahn_JP1977}
J.~Cahn and S.~Allen.
\newblock A microscopic theory for domain wall motion and its experimental
  verification in {F}e-{A}l alloy domain growth kinetics.
\newblock {\em J. de Physique Colloques}, 38:C7--51, 1977.

\bibitem{CahnHilliard_JCP1958}
J.~Cahn and J.~Hilliard.
\newblock Free energy of a nonuniform system. i. interfacial free energy.
\newblock {\em J. Chem. Phys.}, 28:258, 1958.

\bibitem{Cahn_1961}
J.W. Cahn.
\newblock On spinodal decomposition.
\newblock {\em Acta Metallurgica}, 9(9):795--801, 1961.

\bibitem{Zhao_PRE2017}
B.~Camley, Y.~Zhao, B.~Li, H.~Levine, and W.-J. Rappel.
\newblock Crawling and turning in a minimal reaction-diffusion cell motility
  model: coupling cell shape and biochemistry.
\newblock {\em Phys. Rev. E}, 95:012401, 2017.

\bibitem{ChanNejadWei_PhysicaD2019}
H.~Chan, M.~Nejad, and J.~Wei.
\newblock Lamellar phase solutions for diblock copolymers with nonlocal
  diffusions.
\newblock {\em Physica D: Nonlinear Phenomena}, 388:22--32, 2019.

\bibitem{ChenShen_CPC1998}
L.~Q. Chen and J.~Shen.
\newblock Applications of semi-implicit {F}ourier-spectral method to phase
  field equations.
\newblock {\em Comput. Phys. Commun.}, 108:147, 1998.

\bibitem{Chen_Zhao_2022}
L.~Q. Chen and Y.~H. Zhao.
\newblock From classical thermodynamics to phase-field method.
\newblock {\em Progress in Materials Science}, 124:10086, 2022.

\bibitem{Chen_ARMR2002}
L.Q. Chen.
\newblock Phase-field models for microstructure evolution.
\newblock {\em Annu. Rev. Mater. Res.}, 32:113, 2002.

\bibitem{ChengYangShen_JCP2017}
W.~Cheng, X.~Yang, and J.~Shen.
\newblock Efficient and accurate numerical schemes for a hydro-dynamically
  coupled phase field diblock copolymer model.
\newblock {\em J. Comput. Phys.}, 341:44, 2017.

\bibitem{ChoiZhao_DCDSB2021}
H.~Choi and Y.~Zhao.
\newblock Second-order stabilized semi-implicit energy stable schemes for
  bubble assemblies in binary and ternary systems.
\newblock {\em DCDS-B}, 2021.

\bibitem{Choksi_NonlinearScience2001}
R.~Choksi.
\newblock Scaling laws in microphase separation of diblock copolymers.
\newblock {\em J. Nonlinear Sci.}, 11:223--236, 2011.

\bibitem{Choksi_QAM2012}
R.~Choksi.
\newblock On global minimizers for a variational problem with long-range
  interactions.
\newblock {\em Quart. Appl. Math.}, 70:517--537, 2012.

\bibitem{CoxMatthews_JCP2002}
S.~M. Cox and P.~C. Matthews.
\newblock Exponential time differencing for stiff systems.
\newblock {\em J. Comput. Phys.}, 176:430, 2002.

\bibitem{crandall1971bifurcation}
Michael~G Crandall and Paul~H Rabinowitz.
\newblock Bifurcation from simple eigenvalues.
\newblock {\em Journal of Functional Analysis}, 8(2):321--340, 1971.

\bibitem{vanderWaals_1893}
J.D.V. der Waals.
\newblock Theorie thermodynamique de la capillarite, dans l'hypothese d'une
  variation continue de la densite.
\newblock {\em Archives Neerlandaises des sciences exactes et naturelles},
  XXVIII:121--209, 1979.

\bibitem{DuFeng_HNA2020}
Q.~Du and X.~Feng.
\newblock The phase field method for geometric moving interfaces and their
  numerical approximations.
\newblock {\em Handbook of Numerical Analysis}, 21:425, 2020.

\bibitem{DuJuLiQiao_SIAMReview2021}
Q.~Du, L.~Ju, X.~Li, and Z.~Qiao.
\newblock Maximum bound principles for a class of semilinear parabolic
  equations and exponential time-differencing schemes.
\newblock {\em SIAM Review}, 63:317, 2021.

\bibitem{DuNicolaides_SINU1991}
Q.~Du and R.~A. Nicolaides.
\newblock Numerical analysis of a continuum model of phase transition.
\newblock {\em SIAM J. Numer. Anal.}, 28:1310, 1991.

\bibitem{Eyre_MRS1998}
D.~J. Eyre.
\newblock Unconditionally gradient stable time marching the {C}ahn-{H}illiard
  equation.
\newblock {\em MRS Online Proceedings Library Archive}, 529, 1998.

\bibitem{Fix_1983}
G.J. Fix.
\newblock {\em Phase field problems for free boundary problems}.
\newblock Free Boundary Problems: Theory and Applications. Pitman, 1983.

\bibitem{FHB}
A.~Friedman and B.~Hu.
\newblock Bifurcation from stability to instability for a free boundary problem
  arising in a tumor model.
\newblock {\em Archive for rational mechanics and analysis}, 180(2):293--330,
  2006.

\bibitem{FHB1}
A.~Friedman and B.~Hu.
\newblock Bifurcation for a free boundary problem modeling tumor growth by
  stokes equation.
\newblock {\em SIAM Journal on Mathematical Analysis}, 39(1):174--194, 2007.

\bibitem{friedman2007bifurcation}
A.~Friedman and B.~Hu.
\newblock Bifurcation for a free boundary problem modeling tumor growth by
  stokes equation.
\newblock {\em SIAM Journal on Mathematical Analysis}, 39(1):174--194, 2007.

\bibitem{FRF1}
A.~Friedman and F.~Reitich.
\newblock Symmetry-breaking bifurcation of analytic solutions to free boundary
  problems: an application to a model of tumor growth.
\newblock {\em Transactions of the American Mathematical Society},
  353(4):1587--1634, 2001.

\bibitem{friedman2001symmetry}
A.~Friedman and F.~Reitich.
\newblock Symmetry-breaking bifurcation of analytic solutions to free boundary
  problems: an application to a model of tumor growth.
\newblock {\em Transactions of the American Mathematical Society},
  353(4):1587--1634, 2001.

\bibitem{GennipPeletier_CVPDE2008}
Y.~Gennip and M.~Peletier.
\newblock Copolymer-homopolymer blends: global energy minimisation and global
  energy bounds.
\newblock {\em Calc. Var. Partial Differ. Equ.}, 33:75--111, 2008.

\bibitem{Ginzburg_Landau_JETP1950}
V.~Q. Ginzburg and L.~E. Landau.
\newblock On the theory of superconductivity.
\newblock {\em Soviet Physics — JETP}, 20(12):1064--1082, 1950.

\bibitem{goldstein2011cahn}
G.R. Goldstein, A.~Miranville, and G.~Schimperna.
\newblock A cahn--hilliard model in a domain with non-permeable walls.
\newblock {\em Physica D: Nonlinear Phenomena}, 240(8):754--766, 2011.

\bibitem{haber1}
R.~Haber and H.~Unbehauen.
\newblock Structure identification of nonlinear dynamic systems survey on
  input/output approaches.
\newblock {\em Automatica}, 26(4):651--677, 1990.

\bibitem{hao2021adaptive}
W.~Hao.
\newblock An adaptive homotopy tracking algorithm for solving nonlinear
  parametric systems with applications in nonlinear odes.
\newblock {\em Applied Mathematics Letters}, page 107767, 2021.

\bibitem{HCF}
W.~Hao, E.~Crouser, and A.~Friedman.
\newblock Mathematical model of sarcoidosis.
\newblock {\em Proceedings of the National Academy of Sciences},
  111(45):16065--16070, 2014.

\bibitem{HF}
W.~Hao and A.~Friedman.
\newblock The ldl-hdl profile determines the risk of atherosclerosis: a
  mathematical model.
\newblock {\em PloS one}, 9(3):e90497, 2014.

\bibitem{HHHLSZ}
W.~Hao, J.~Hauenstein, B.~Hu, Y.~Liu, A.~Sommese, and Y.-T. Zhang.
\newblock Multiple stable steady states of a reaction-diffusion model on
  zebrafish dorsal-ventral patterning.
\newblock {\em Discrete and Continuous Dynamical Systems-Series S},
  4(6):1413--1428, 2011.

\bibitem{HHHS}
W.~Hao, J.~Hauenstein, B.~Hu, and A.~Sommese.
\newblock A three-dimensional steady-state tumor system.
\newblock {\em Applied Mathematics and Computation}, 218(6):2661--2669, 2011.

\bibitem{HHSSXZ}
W.~Hao, J.~Hauenstein, C.-W. Shu, A.~Sommese, Z.~Xu, and Y.-T. Zhang.
\newblock A homotopy method based on weno schemes for solving steady state
  problems of hyperbolic conservation laws.
\newblock {\em Journal of Computational Physics}, 250:332--346, 2013.

\bibitem{hao2012bifurcation}
W.~Hao, J.~D Hauenstein, B.~Hu, Y.~Liu, A.~J Sommese, and Y.-T. Zhang.
\newblock Bifurcation for a free boundary problem modeling the growth of a
  tumor with a necrotic core.
\newblock {\em Nonlinear Analysis: Real World Applications}, 13(2):694--709,
  2012.

\bibitem{HNS}
W.~Hao, R.~Nepomechie, and A.~Sommese.
\newblock Completeness of solutions of bethe's equations.
\newblock {\em Physical Review E}, 88(5):052113, 2013.

\bibitem{HNS1}
W.~Hao, R.~Nepomechie, and A.~Sommese.
\newblock Singular solutions, repeated roots and completeness for higher-spin
  chains.
\newblock {\em Journal of Statistical Mechanics: Theory and Experiment},
  2014(3):P03024, 2014.

\bibitem{hao2020spatial}
W.~Hao and C.~Xue.
\newblock Spatial pattern formation in reaction--diffusion models: a
  computational approach.
\newblock {\em Journal of Mathematical Biology}, pages 1--23, 2020.

\bibitem{hao2020adaptive}
W.~Hao and C.~Zheng.
\newblock An adaptive homotopy method for computing bifurcations of nonlinear
  parametric systems.
\newblock {\em Journal of Scientific Computing}, 82(3):1--19, 2020.

\bibitem{hao2020bifurcation}
W.~Hao and C.~Zheng.
\newblock Bifurcation analysis of a free boundary model of the atherosclerotic
  plaque formation associated with the cholesterol ratio.
\newblock {\em Chaos: An Interdisciplinary Journal of Nonlinear Science},
  30(9):093113, 2020.

\bibitem{hao2021stochastic}
Wenrui Hao and Chunyue Zheng.
\newblock A stochastic homotopy tracking algorithm for parametric systems of
  nonlinear equations.
\newblock {\em Journal of Scientific Computing}, 87(3):1--14, 2021.

\bibitem{JooXuZhao_IFB2021}
S.~Joo, X.~Xu, and Y.~Zhao.
\newblock Analysis and computation for {A}llen-{C}ahn-{O}hta-{N}akazawa model
  in ternary system.
\newblock {\em Interfaces Free Bound.}, 23:535--559, 2021.

\bibitem{KesslerKoplikLevine_PRA1984}
D.~Kessler, J.~Koplik, and H.~Levine.
\newblock Geometrical models of interface evolution. ii.
\newblock {\em Phy. Rev. A}, 30:3161, 1984.

\bibitem{KesslerKoplikLevine_PRA1985}
D.~Kessler, J.~Koplik, and H.~Levine.
\newblock Geometrical models of interface evolution. iii.
\newblock {\em Phy. Rev. A}, 31:1712, 1985.

\bibitem{knopf2021phase}
P.~Knopf, K.F. Lam, C.~Liu, and S.~Metzger.
\newblock Phase-field dynamics with transfer of materials: The cahn--hilliard
  equation with reaction rate dependent dynamic boundary conditions.
\newblock {\em ESAIM: Mathematical Modelling and Numerical Analysis},
  55(1):229--282, 2021.

\bibitem{LiuShen_PhysicaD}
C.~Liu and J.~Shen.
\newblock A phase field model for the mixture of two incompressible fluids and
  its approximation by a fourier-spectral method.
\newblock {\em Physica D}, 179:211, 2003.

\bibitem{liu2019energetic}
C.~Liu and H.~Wu.
\newblock An energetic variational approach for the cahn--hilliard equation
  with dynamic boundary condition: model derivation and mathematical analysis.
\newblock {\em Archive for Rational Mechanics and Analysis}, 233(1):167--247,
  2019.

\bibitem{MorganSommese1}
A.~Morgan and A.~Sommese.
\newblock Computing all solutions to polynomial systems using homotopy
  continuation.
\newblock {\em Applied Mathematics and Computation}, 24(2):115--138, 1987.

\bibitem{OhtaNakazawa_Macromolecules1993}
H.~Nakazawa and T.~Ohta.
\newblock Microphase separation of {ABC}-type triblock copolymers.
\newblock {\em Macromolecules}, 26:5503–5511, 1993.

\bibitem{OhtaKawasaki_Macromolecules1986}
T.~Ohta and K.~Kawasaki.
\newblock Equilibrium morphology of block copolymer melts.
\newblock {\em Macromolecules}, 19:2621--2632, 1986.

\bibitem{RenTruskinovsky_Elasticity2000}
X.~Ren and L.~Truskinovsky.
\newblock Finite scale microstructures in nonlocal elasticity.
\newblock {\em J. Elasticity}, 59:319--355, 2000.

\bibitem{RenWei_SIAM2000}
X.~Ren and J.~Wei.
\newblock On the multiplicity of solutions of two nonlocal variational
  problems.
\newblock {\em SIAM J. Math. Ana.}, 4:909--924, 2000.

\bibitem{RenWei_JNS2003}
X.~Ren and J.~Wei.
\newblock Triblock copolymer theory: {O}rdered {ABC} lamellar phase.
\newblock {\em J. Nonlinear Sci.}, 13:175--208, 2003.

\bibitem{RenWei_ARMA2013}
X.~Ren and J.~Wei.
\newblock A double bubble in a ternary system with inhibitory long range
  interaction.
\newblock {\em Arch. Ration. Mech. Anal.}, 208:201--253, 2013.

\bibitem{RenWei_ARMA2015}
X.~Ren and J.~Wei.
\newblock A double bubble assembly as a new phase of a ternary inhibitory
  system.
\newblock {\em Arch. Ration. Mech. Anal.}, 215:967--1034, 2015.

\bibitem{Rhein1}
W.~Rheinboldt.
\newblock Numerical methods for a class of finite dimensional bifurcation
  problems.
\newblock {\em SIAM Journal on Numerical Analysis}, 15(1):1--11, 1978.

\bibitem{Rhein2}
W.~Rheinboldt.
\newblock Numerical analysis of continuation methods for nonlinear structural
  problems.
\newblock {\em Computers \& Structures}, 13(1):103--113, 1981.

\bibitem{Rowlinson_1979}
J.S. Rowlinson.
\newblock Translation of vanderwaals,jd the thermodynamic theory of capillarity
  under the hypothesis of a continuous variation of density.
\newblock {\em Journal of Statistical Physics}, 20(2):197--244, 1979.

\bibitem{ShenXuYang_JCP2018}
J.~Shen, J.~Xu, and J.~Yang.
\newblock The scalar auxiliary variable (sav) approach for gradient flows.
\newblock {\em J. Comput. Phys.}, 353:407, 2018.

\bibitem{SongShu_ISC2017}
H.~Song and C.-W. Shu.
\newblock Unconditional energy stability analysis of a second order
  implicit–explicit local discontinuous {G}alerkin method for the
  {C}ahn–{H}illiard equation.
\newblock {\em J. Sci. Comput.}, 73:1178, 2017.

\bibitem{Steinbach_2013}
I.~Steinbach.
\newblock Phase-field model for microstructure evolution at the mesoscopic
  scale.
\newblock {\em Annual Review of Materials Research}, 43:89--107, 2013.

\bibitem{Stephen_Suhl_1964}
M.~J. Stephen and H.~Suhl.
\newblock Weak time dependence in pure superconductors.
\newblock {\em Phys. Rev. Lett.}, 13(26):797--800, 1964.

\bibitem{WangRenZhao_CMS2019}
C.~Wang, X.~Ren, and Y.~Zhao.
\newblock Bubble assemblies in ternary systems with long range interaction.
\newblock {\em Comm. Math. Sci.}, 17:2309--2324, 2019.

\bibitem{Boettinger_Warren_Beckerman_Karma_2002}
C.~Beckermann W.J.~Boettinger, J.A.~Warren and A.~Karma.
\newblock Phase-field simulation of solidification.
\newblock {\em Annual Review of Materials Research}, 32:163--194, 2002.

\bibitem{Xu_CMAME2019}
J.~Xu, Y.~Li, S.~Wu, and A.~Bousquetd.
\newblock On the stability and accuracy of partially and fully implicit schemes
  for phase field modeling.
\newblock {\em Comput. Methods Appl. Mech. Eng.}, 345:826, 2019.

\bibitem{XuZhao_JSC2019}
X.~Xu and Y.~Zhao.
\newblock Energy stable semi-implicit schemes for
  {A}llen-{C}ahn-{O}hta-{K}awasaki model in binary system.
\newblock {\em J. Sci. Comput.}, 80:1656--1680, 2019.

\bibitem{XuZhao_JSC2020}
X.~Xu and Y.~Zhao.
\newblock Maximum principle preserving schemes for binary systems with
  long-range interactions.
\newblock {\em J. Sci. Comput.}, 84:33, 2020.

\bibitem{XuDu_JNA2021}
Z.~Xu and Q.~Du.
\newblock On the ternary ohta-kawasaki free energy and its one dimensional
  global minimizers.
\newblock {\em arXiv:2111.09877}, 2021.

\bibitem{Yang_JCP2016}
X.~Yang.
\newblock Linear, first and second-order, unconditionally energy stable
  numerical schemes for the phase field model of homopolymer blends.
\newblock {\em J. Comput. Phys.}, 327:294, 2016.

\bibitem{zhao2020symmetry}
Xinyue~Evelyn Zhao and Bei Hu.
\newblock Symmetry-breaking bifurcation for a free-boundary tumor model with
  time delay.
\newblock {\em Journal of Differential Equations}, 269(3):1829--1862, 2020.

\bibitem{zhao2021bifurcation}
Xinyue~Evelyn Zhao and Bei Hu.
\newblock Bifurcation for a free boundary problem modeling a small arterial
  plaque.
\newblock {\em Journal of Differential Equations}, 288:250--287, 2021.

\bibitem{Aranson2016}
Falko Ziebert and Igor~S. Aranson.
\newblock Computational approaches to substrate-based cell motility.
\newblock {\em npj Computational materials}, 2:16019, 2016.

\end{thebibliography}

\end{document}